\documentclass[11pt]{amsart}
\usepackage{graphicx,amsmath,amsthm,amsfonts}

\newcommand{\w}[1]{\ensuremath{\mathbf{\omega}_{#1}}}

\newtheorem{prop}{Proposition}[section]
\newtheorem{lemma}{Lemma}[section]

\newtheorem{conj}{Conjecture}[section]

\theoremstyle{definition}
\newtheorem{defn}{Definition}[section]
\newtheorem{query}{Query}[section]
\newtheorem*{rem}{Remark}

\newcommand{\p}[2]{\ensuremath{p^{#1}_{#2}}}
\newcommand{\q}[2]{\ensuremath{q^{#1}_{#2}}}

\newcommand{\MM}[2]{\ensuremath{\mathcal{M}^{#1}_{#2}}}
\newcommand{\LL}[2]{\ensuremath{\mathcal{L}^{#1}_{#2}}}

\newcommand{\CC}[1]{\ensuremath{\mathbf{C}^{#1}}}
\newcommand{\CP}[1]{\ensuremath{\mathbf{CP}^{#1}}}
\newcommand{\PP}[1]{\ensuremath{\mathbf{P}^{#1}}}
\newcommand{\R}[1]{\ensuremath{\mathbf{R}^{#1}}}
\newcommand{\RP}[1]{\ensuremath{\mathbf{RP}^{#1}}}

\newcommand{\RR}{\ensuremath{\mathcal{R}}}
\newcommand{\HH}{\ensuremath{\mathcal{H}}}
\newcommand{\T}{\ensuremath{\mathcal{T}}}

\newcommand{\Alt}[1]{\ensuremath{\mathcal{A}_{#1}}}
\newcommand{\Sym}[1]{\ensuremath{\mathcal{S}_{#1}}}

\newcommand{\G}[1]{\ensuremath{\mathcal{G}_{#1}}}
\newcommand{\Z}[1]{\ensuremath{\mathbf{Z}_{#1}}}

\newcommand{\Ss}{\ensuremath{\mathcal{S}}}

\newcommand{\obar}[1]{\ensuremath{\overline{#1}}}
\newcommand{\tld}[1]{\ensuremath{\widetilde{#1}}}

\newcommand{\arrGap}{5pt}
\newcommand{\tabGap}{8pt}

\newcommand{\fig}[1]{Figure~\ref{#1}}

\newcommand{\app}[1]{\textsc{Appendix~\ref{#1}}}

\newcommand{\tab}[1]{\textsc{Table~\ref{#1}}}

\newcommand{\basinWidth}{3.6in}

\begin{document}

\title[Solving the octic by iteration]
 {Solving the octic by iteration in six dimensions}
\author{Scott Crass}
\address{
Mathematics Department\\
California State University, Long Beach\\
Long Beach, CA 90840\\}
\date{\today}
\email{scrass@csulb.edu}

\begin{abstract}

The requirement for solving a polynomial is a means of breaking its
symmetry, which in the case of the octic, is that of the symmetric group
\Sym{8}. Its eight-dimensional linear permutation representation restricts
to a six-dimensional projective action. A mapping of complex projective
$6$-space with this \Sym{8} symmetry can provide the requisite
symmetry-breaking tool.

The article describes some of the \Sym{8} geometry in \CP{6} as
well as a special \Sym{8}-symmetric rational map in degree four.
Several basins-of-attraction plots illustrate the map's geometric
and dynamical properties. The work culminates with an explicit
algorithm that uses this map to solve a general octic. A
concluding discussion treats the generality of this approach to
equations in higher degree.

\end{abstract}
\maketitle

\section{Overview}

In \cite{quintic}, I develop a solution to the quintic that relies
on a single iteration in three dimensions.  Given almost any
quintic $p$, there is a map on complex projective $3$-space \CP{3}
whose dynamics provides for a root of $p$. This approach is
geometric: the map has the \Sym{5} symmetry of the general
quintic.

The present paper extends this work to the eighth-degree equation.
At its core is an \Sym{8}-symmetric map on \CP{6} whose geometric
behavior is connected to a special configuration of lines.
Motivating this general project is a desire to develop solutions
to equations that utilize symmetrical and geometrically elegant
dynamical systems.  I do not consider the numerical aspects of the
algorithm.  However, since the map converges very rapidly,
numerical considerations might well be of interest.  Indeed, the
geometry associated with the map accounts for the rapid
convergence.

In addition, the work establishes the existence of a method for
all degrees greater than four that is analogous to the
eighth-degree case. This involves showing that there is an
infinite family of maps---one for each dimension greater than
two---with special geometric properties.

Finally, these maps add to the examples of complex dynamics in
several dimensions.  This recently-active and difficult field
seems to be in need of examples that are not concocted for
purposes of illustration.

The work unfolds in four stages: 1) some background geometry, 2) a
special map with \Sym{8} symmetry, 3) a solution to the octic
based on the preceding stages, 4) a consideration of whether the
octic algorithm generalizes to higher degree equations.

\textbf{Section \ref{sec:geom}: \Sym{8} geometry.} The setting
here is \CP{6} upon which the symmetric group \Sym{8} acts.
Finding a map with special \Sym{8} geometry requires some
familiarity with this action. We will consider some features
associated with the map that emerges in the second stage.  Indeed,
the discovery of this map derives from an awareness of the
algebraic and geometric surroundings:

\begin{itemize}

\item coordinate systems on \CP{6}

\item the structure of certain special orbits of points, lines, planes, and
hyperplanes

\item the system of \Sym{8}-invariant polynomials---the building-blocks
for maps that are \Sym{8}-symmetric.

\end{itemize}

\textbf{Section \ref{sec:maps}: Maps with \Sym{8} symmetry.} At
this stage, we exploit our geometric understanding to discover
empirically a map associated with the complete graph on eight
vertices---an $8$-point \Sym{8} orbit. The discussion turns to its
geometric and dynamical behavior---empirical testing suggests that
the $8$-point orbit is the only attractor. However, whether it
possesses this or another desired global dynamical property is not
known. In light of substantial experimental and graphical
evidence, I attribute these properties to the map in conjectures.

\textbf{Section \ref{sec:octic}: Dynamical solution to the octic.}
A special family of octics corresponds to a \emph{rigid} family
$\mathcal{E}$ of \Sym{8}-symmetric maps on \CP{6}. `Rigidity' means that
each member of $\mathcal{E}$ is conjugate to a single reference map $f$.
Thus, associated with an octic $p$ is a map $g_p=\phi_p\,f\,\phi_p^{-1}$
that we iterate. Using \Sym{8} tools, the dynamical
output---conjecturally, a single \Sym{8} orbit---provides for an
approximate solution to $p=0$.  Since almost any octic $p$ transforms into
the special family, the solution is general.

Note: Up to this point, the exposition follows that of
\cite{quintic} which the reader can consult for details.

\textbf{Section \ref{sec:gen}: Generalization: Solving the
$n$th-degree equation by iteration in $n-2$ dimensions.}  The
geometric and dynamical description of the octic-solving map has
an analogue for each permutation-based \Sym{n} action on \CP{n-2}
with $n \geq 5$. Here, we can show that there is always a map with
the desired special properties. Given such a map for which the
$n$-point orbit is the attractor, the solution algorithm for the
octic generalizes to one for the $n$th degree equation.

\section{\Sym{8} acts on \CP{6}} \label{sec:geom}

The search for a special \Sym{8}-symmetric map begins with a faithful
action of \Sym{8}.  Klein's approach to the $n$th degree equation was to
look for the lowest dimensional faithful action of \Sym{n} or the
alternating group \Alt{n}. For $n<8$, there are special actions of either
\Sym{n} or \Alt{n}; that is, there are faithful representations that do
not derive directly from permutations on \CC{n}.  However, special
geometry---at least for linear actions---ends at $n=7$.  When $n>7$, the
space of least dimension on which \Sym{n} or \Alt{n} acts faithfully is
\CP{n-2}. \cite{Wiman}

The permutation action of the symmetric group \Sym{8} on \CC{8}
preserves the hyperplane
$$\HH_x = \Biggl\{\sum_{k=1}^8 x_k =0 \Biggr\} \simeq \CC{7}$$
and, thereby, restricts to a faithful seven-dimensional
irreducible representation. This \CC{7} action projects one-to-one
to a group \G{8!} on \CP{6}.

\subsection{Coordinates}

For many purposes, the most perspicuous geometric description of
\G{8!} employs eight coordinates that sum to zero. One advantage
is the simple expression of the \Sym{8}-duality between points and
hyperplanes. In general, for a finite action \G{} whose matrix
representatives are unitary, a point $a$ is \G{}-\emph{dual} to
the hyperplane
$$\LL{}{} = \{ \obar{a} \cdot x = 0 \}.$$
Consequently, $a$ and \LL{}{} have the same stabilizer in \G{}.
Since the action of \Sym{8} on \CC{7} is orthogonal, a point
$$a = [a_1,a_2,a_3,a_4,a_5,a_6,a_7,a_8]_{\sum a_k = 0} \in \CP{6}$$
corresponds to the hyperplane
$$\{a \cdot x=0\}=\Biggl\{\sum_{k=1}^8 a_k\,x_k = 0\Biggr\}.$$
(Square brackets indicate homogeneous coordinates.)

A system of seven \emph{hyperplane coordinates} describes the hyperplane
$\HH_u$. It arises from the ``hermitian" change of variable
$$u = H\,x  \qquad x = \obar{H^T} u \qquad
H=\frac{1}{2\,\sqrt{2}}
\begin{pmatrix}
1&\w{}&i&\w{}^3&-1&\w{}^5&-i&\w{}^7\\
1&i&-1&-i&1&i&-1&-i\\
1&\w{}^3&-i&\w{}&-1&\w{}^7&i&\w{}^5\\
1&-1&1&-1&1&-1&1&-1\\
1&\w{}^5&i&\w{}^7&-1&\w{}&-i&\w{}^3\\
1&-i&-1&i&1&-i&-1&i\\
1&\w{}^7&-i&\w{}^5&-1&\w{}^3&i&\w{}
\end{pmatrix}
$$
where $\w{}=e^{\pi\,i/4}$ and the choice of scalar factor gives
\begin{equation}
H\,\obar{H^T} = I_7 \qquad \obar{H^T} H = (a_{ij}) \qquad a_{ij} =
\begin{cases} -\frac{7}{8}&i=j\\ \frac{1}{8}&i\neq j \end{cases}
\qquad i,j=1, \dots 8.
\end{equation}

\subsection{Invariant polynomials}

According to the fundamental result on symmetric functions the $n$
elementary symmetric functions of degrees one through $n$ generate
the ring of \Sym{n}-invariant polynomials. Since the \Sym{8}
action on \CP{6} occurs where the degree-$1$ symmetric polynomial
vanishes, there are seven generating \G{8!}-invariants. By
Newton's identities, the power sums
$$F_k(x) = \sum_{\ell=1}^8 x_{\ell}^k \qquad k=2, \dots, 8$$
also generate the \G{8!} invariants. In hyperplane coordinates,
the \emph{forms} in degrees two and three are \small
\begin{align*}
\Phi_2(u)=&\ F_2(\obar{H^T} u)=\ u_4^2 + 2\,u_3\,u_5 + 2\,u_2\,u_6 + 2\,u_1\,u_7\\[\arrGap]
\Phi_3(u)=&\ \frac{3}{2\,\sqrt{2}}\,\bigl( u_2\,u_3^2 + u_2^2\,u_4
+ 2\,u_1\,u_3\,u_4 + 2\,u_1\,u_2\,u_5 + u_1^2\,u_6 + u_5^2\,u_6 +
u_4\,u_6^2\\
&\ + 2\,u_4\,u_5\,u_7 + 2\,u_3\,u_6\,u_7 + u_2\,u_7^2 \bigr)
\end{align*}
\normalsize The remaining five generating invariants $\Phi_k(u)$
arise algebraically from these two. Classical techniques show that
a \emph{relative invariant}---invariant up to a multiplicative
character---results from taking the determinant of the
\emph{bordered hessian}
$$BH(F,G,J) =
\begin{pmatrix}
&&&\frac{\partial{G}}{\partial{x_1}}\\
&H(F)&&\vdots\\
&&&\frac{\partial{G}}{\partial{x_n}}\\
\frac{\partial{J}}{\partial{x_1}}& \dots &
 \frac{\partial{J}}{\partial{x_n}}&0
\end{pmatrix}$$
of three $\mathrm{GL}_n(\CC{})$ invariants $F$, $G$, and $J$.  The
$n\times n$ matrix $H(F)$ is the hessian of $F$.

\begin{prop} \label{prop:hess}
Given $T \in \mathrm{GL}_n(\CC{})$ and
$B(x)=\det{(BH(F(x),G(x),J(x)))}$ for invariants $F, G, J$,
\begin{align*}
B(Tx) &= (\det{T})^{-2} B(x).
\end{align*}
\end{prop} \noindent
For the permutation action of \Sym{8}, this results in an
\emph{absolute} invariant that is expressible in terms of
the generators $\Phi_k$. The following result will serve a
subsequent computational purpose. (Many of this work's results
arise from \emph{Mathematica} computations.)

\begin{prop} \label{prop:PhiInG}

With
\begin{align*}
G_4 &= \det{(BH(\Phi_2,\Phi_3,\Phi_3))}&
G_5 &= \det{(BH(\Phi_2,\Phi_3,\Phi_4))}\\
G_6 &= \det{(BH(\Phi_2,\Phi_4,\Phi_4))}&
G_7 &= \det{(BH(\Phi_2,\Phi_4,\Phi_5))}\\
G_8 &= \det{(BH(\Phi_2,\Phi_5,\Phi_5))},
\end{align*}
the ``power-sum" invariants are given by
\begin{align*}
\Phi_4 &= \frac{1}{576} \bigl(72\,\Phi_2^2 + G_4 \bigr)&
\Phi_5 &= \frac{1}{768} \bigl(96\,\Phi_2\,\Phi_3 + G_5 \bigr)\\
\Phi_6 &= \frac{1}{960} \bigl(120\,\Phi_2\,\Phi_4 + G_6 \bigr)&
\Phi_7 &= \frac{1}{1280} \bigl(160\,\Phi_3\,\Phi_4 + G_7 \bigr)\\
\Phi_8 &= \frac{1}{1600} \bigl(200\,\Phi_4^2 + G_8 \bigr).
\end{align*}

\end{prop}

\subsection{Special orbits}  \label{sec:specOrbs}

The $6$-dimensional \Sym{8} action comes in both real and complex
versions. This means that, in the standard $x$ coordinates, \G{8!}
acts on \RR---the \RP{6} of points with real components.
\tab{tab:RP6} enumerates some special orbits contained in \RR. For
ease of expression, I will refer to special points (or lines,
planes, etc.) in terms of the orbit size: ``$8$-points"
($28$-lines, $56$-planes, $28$-hyperplanes). Also, these points
receive a symbolic description in reference to orbit size
(superscript) and coordinate expression (subscript).

Corresponding to each special point $a$ is the hyperplane $\{a
\cdot x=0\}$. In the case of the $28$-points
$$[1,-1,0,0,0,0,0,0],\ \dots\ ,[0,0,0,0,0,0,1,-1],$$
there are the $28$-hyperplanes
$$\LL{5}{28_{12}}=\{x_1 = x_2\},\ \dots\ ,\LL{5}{28_{78}}=\{x_7 = x_8\}.$$
The involutions
$$x_1 \leftrightarrow x_2,\ \dots\ , x_7 \leftrightarrow x_8$$
pointwise fix the respective hyperplanes. These $28$
transpositions generate \G{8!} so that it acts as both a
\emph{real} and \emph{complex reflection group}. (See \cite{ST}.)

Various special planes and lines appear as intersections of the
$28$-hyperplanes. \tab{tab:planes} and \tab{tab:lines} summarize the
situation. Of particular dynamical significance is the collection of
$28$-lines \LL{1}{28_{ij}}. This configuration forms the complete graph on
the $8$-points.  (See \fig{fig:28Lines} for two views.)

\begin{figure}[h]

\resizebox{2.9in}{!}{\includegraphics{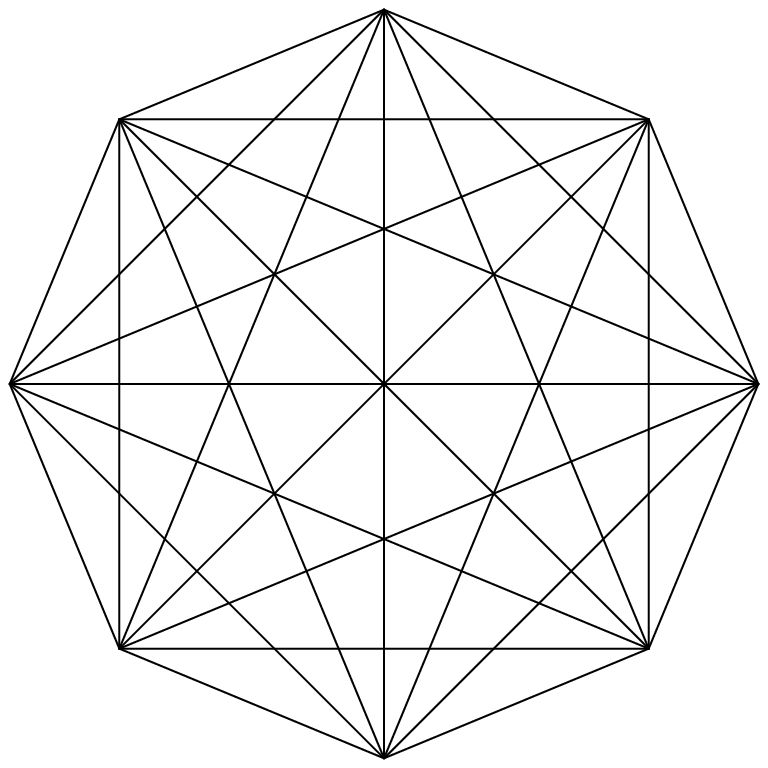}}
\resizebox{2.9in}{!}{\includegraphics{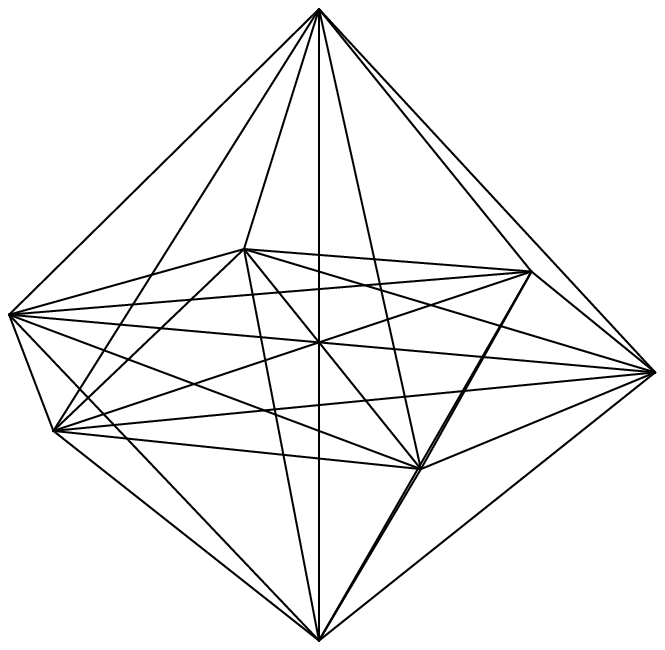}}

\caption{Configuration of $28$-lines and $8$-points}

\label{fig:28Lines} \flushleft

\end{figure}

\begin{table}[h]

\caption{Three special orbits}

\label{tab:RP6}

$$\begin{array}{c|c|c|c}

\text{Size}&\text{Representative}&
\text{Descriptor}&\text{Stabilizer}\\
\hline
&&&\\
8&[-7,1,1,1,1,1,1,1]&\p{8}{1}&\Sym{7}\\[\tabGap]

28&[0,0,0,0,0,0,1,-1]&\p{28}{78}&\Sym{6} \times \Z{2}\\[\tabGap]

28&[1,1,1,1,1,1,-3,-3]&\q{28}{78}&\Sym{6} \times \Z{2}\\[\tabGap]

\end{array}$$

\end{table}
\begin{table}[h]

\caption{Some fundamental \CP{2} orbits}

\label{tab:planes}

$$\begin{array}{c|c|c|c|c}

&&\text{Set-wise}&\text{Point-wise}&\text{Restricted}\\
\text{Geometric definition}&\text{Descriptor}&
\text{stabilizer}&\text{stabilizer}&\text{action}\\
\hline
&&&&\\
\LL{5}{28_{ij}} \cap \LL{5}{28_{ik}}\cap \LL{5}{28_{i\ell}}\cap
\LL{5}{28_{im}}
&\LL{2}{56}&\Sym{5} \times \Sym{3}&\Sym{5}&\Sym{3}\\[\tabGap]

\LL{5}{28_{ij}} \cap \LL{5}{28_{k\ell}}\cap \LL{5}{28_{mn}}\cap
\LL{5}{28_{pq}}
&\LL{2}{105}&\Sym{4} \times \Z{2}^4&\Z{2}^4&\Sym{4}\\[\tabGap]

\LL{5}{28_{ij}} \cap \LL{5}{28_{ik}}\cap \LL{5}{28_{\ell m}}\cap
\LL{5}{28_{\ell n}}
&\LL{2}{280}&\Sym{3}^2 \times \Z{2}^2&\Sym{3}^2&\Z{2}^2\\[\tabGap]

\LL{5}{28_{ij}} \cap \LL{5}{28_{k\ell}}\cap \LL{5}{28_{km}}\cap
\LL{5}{28_{kn}}
&\LL{2}{420}&\Sym{4} \times \Z{2}^2&\Sym{4} \times \Z{2}&\Z{2}\\[\tabGap]

\LL{5}{28_{ij}} \cap \LL{5}{28_{k\ell}}\cap \LL{5}{28_{mn}}\cap
\LL{5}{28_{mp}} &\LL{2}{840}&\Sym{3} \times \Z{2}^3&\Sym{3} \times
\Z{2}^2&\Z{2}

\end{array}$$

\end{table}
\begin{table}[h]

\caption{Special \CP{1} orbits}

\label{tab:lines}

$$\begin{array}{c|c|c|c|c}

&&\text{Set-wise}&\text{Point-wise}&\text{Restricted}\\
\text{Geometric definition}&\text{Descriptor}&
\text{stabilizer}&\text{stabilizer}&\text{action}\\
\hline
&&&&\\
\LL{5}{28_{ij}} \cap \LL{5}{28_{ik}}\cap \LL{5}{28_{i\ell}}\cap
 \LL{5}{28_{im}}\cap \LL{5}{28_{in}}
&\LL{1}{28_{pq}}&\Sym{6} \times \Z{2}&\Sym{6}&\Z{2}\\
&\scriptstyle{p,q \neq i,\dots,n}&&&\\[\tabGap]

\LL{5}{28_{ij}} \cap \LL{5}{28_{k\ell}}\cap \LL{5}{28_{km}}\cap
 \LL{5}{28_{kn}}\cap \LL{5}{28_{kp}}
&\LL{1}{168}&\Sym{5} \times \Z{2}&\Sym{5} \times \Z{2}&\Z{1}\\[\tabGap]

\LL{5}{28_{ij}} \cap \LL{5}{28_{k\ell}}\cap \LL{5}{28_{mn}}\cap
 \LL{5}{28_{mp}}\cap \LL{5}{28_{mq}}
&\LL{1}{210}&\Sym{4} \times \Z{2}^3&\Sym{4} \times \Z{2}^2&\Z{2}\\[\tabGap]

\LL{5}{28_{ij}} \cap \LL{5}{28_{k\ell}}\cap \LL{5}{28_{km}}\cap
 \LL{5}{28_{np}}\cap \LL{5}{28_{nq}}
&\LL{1}{280}&\Sym{3}^2 \times \Z{2}^2&\Sym{3} \times \Z{2}&\Z{2}\\[\tabGap]

\LL{5}{28_{ij}} \cap \LL{5}{28_{ik}}\cap \LL{5}{28_{\ell m}}\cap
 \LL{5}{28_{\ell n}}\cap \LL{5}{28_{\ell p}}
&\MM{1}{280}&\Sym{4} \times \Sym{3}&\Sym{4} \times \Sym{3}&\Z{1}\\[\tabGap]

\end{array}$$

\end{table}

\section{Equivariant maps} \label{sec:maps}

The primary tool to be used in solving the general octic is a
rational map $f$ with \Sym{8} symmetry. In algebraic terms, this
means that
$$f \circ T = T\,f \quad \text{for all}\ T \in \G{8!}$$
while the geometric upshot is that $f$ sends \G{8!}-orbits to
\G{8!}-orbits. Furthermore, the map should have \emph{reliable
dynamics}: its attractor

\begin{enumerate}

\item[1)] is a \emph{single} \G{8!} orbit

\item[2)] has a corresponding basin with full measure in \CP{6}
(\emph{strongly} reliable)

\item[$2'$)] alternatively, has a corresponding basin that is dense in
\CP{6} (\emph{weakly} reliable).

\end{enumerate}

\subsection{Basic maps}

A finite group action \G{} on \CC{n} induces an action on the
associated exterior algebra. Moreover, \G{}-invariant
$(n-1)$-forms correspond to \G{}-equivariant maps.  (See
\cite{sextic}.)

For a reflection group, the number of generating $0$-forms (that
is, polynomials) is the dimension of the action. \cite[p. 282]{ST}
From a result in complex reflection groups, this is also the
number of generating $1$-forms and $(n-1)$-forms. \cite[p.
232]{OT} Indeed, the generating $1$-forms are exterior derivatives
of the $0$-forms while the generating $(n-1)$-forms are wedge
products of $1$-forms.

\begin{prop} \label{prop:genMaps}

With $X^k_i = - 7\,x_i^k + \sum_{j\neq i} x_j^k $, the seven maps
$$
f_k(x) =
\bigl[X^k_1,X^k_2,X^k_3,X^k_4,X^k_5,X^k_6,X^k_7,X^k_8\bigr] \qquad
k=1, \dots, 7
$$
generate the module of \G{8!} equivariants over the ring of \G{8!}
invariants.

\end{prop} \noindent

These maps are projections onto the hyperplane $\HH_x$ along
$[1,1,1,1,1,1,1,1]$ of the power maps
$$\bigl[x_1^k,x_2^k,x_3^k,x_4^k,x_5^k,x_6^k,x_7^k,x_8^k\bigr].$$

\begin{prop}

Under an orthogonal action an invariant $F(x)$ gives rise to an
equivariant $f(x)$ by means of a formal gradient
$$f(x) = \nabla_x F(x) =
\biggl[\frac{\partial F}{\partial x_1}(x),\dots,
 \frac{\partial F}{\partial x_n}(x)\biggr].$$

\end{prop}

\begin{proof}
See \cite{quintic}.
\end{proof}

Note that the \G{8!}-equivariant $f_k(x)$ is \emph{not} equal to $\nabla_x
F_{k+1}(x)$, but is a multiple of
$$\nabla_x F_{k+1}(x)|_{x_i^k \rightarrow X_i^k}.$$
While this may be a source of confusion, it does not cause problems since
we are working on the hyperplane $\HH_x$. When using hyperplane
coordinates on $\HH_u$, the discrepancy manifests itself in the appearance
of a constant $-\tfrac{8}{k+1}$ for each map $\phi_k(u)$. (See below.)

Recalling the change of coordinates from $\HH_x$ to $\HH_u$, a map on
$\HH_x$ expresses itself as a map
$$\phi(u) = H f(\obar{H^T} u)$$
on $\HH_u$.  Having these maps in terms of the basic $u$-invariants
$\Phi_k(u)$ will be useful.

\begin{defn} \label{df:R}

Let
$$
R = (r_{ij}) =
\begin{cases} 1 & i+j=8\\0 & \text{otherwise} \end{cases}
\quad \text{and} \quad \nabla_u^r F(u) = R\,\nabla_u F(u)
$$
represent the \emph{reversed identity} and \emph{reversed gradient}.

\end{defn}

\begin{prop} \label{prop:Hmaps}

In $\HH_u$ coordinates, the map $\phi(u) = H\,f(\obar{H^T} u)$ is
given by
$$\phi(u) = \nabla_u^r \Phi(u)$$
where $\Phi(u) = F(\obar{H^T} u) = F(x)$ and $f(x)=\nabla_x F(x)$.

\end{prop}

\begin{proof}
\cite{quintic}.
\end{proof}

Thus, the generating equivariants $\phi_k(u) = H f_k(\obar{H^T}
u)$ are
\begin{equation} \label{eq:equiv}
\phi_k(u) = -\frac{8}{k+1}\,\nabla_u^r \Phi_{k+1}(u).
\end{equation}
Although the factors $-\tfrac{8}{k+1}$ have no projective effect on the
maps individually, they do play a role when forming combinations of maps
from parametrized families that do not have a common degree.  (See
Section~\ref{sec:S8equivFam}.) Explicit expressions for the maps of
degrees one and two are
\begin{align*}
\phi_1(u) =&\ -8\,[u_1,u_2,u_3,u_4,u_5,u_6,u_7] \\
\phi_2(u) =&\ -2\,\sqrt{2}\,\bigl[ 2\,( u_4\,u_5 + u_3\,u_6 +
u_2\,u_7 ),
u_1^2 + u_5^2 + 2\,u_4\,u_6 + 2\,u_3\,u_7,\\
&\ 2\,( u_1\,u_2 + u_5\,u_6 + u_4\,u_7 ), u_2^2 + 2\,u_1\,u_3 +
u_6^2 + 2\,u_5\,u_7,
2\,( u_2\,u_3 + u_1\,u_4 + u_6\,u_7 ),\\
&\ u_3^2 + 2\,u_2\,u_4 + 2\,u_1\,u_5 + u_7^2, 2\,( u_3\,u_4 +
u_2\,u_5 + u_1\,u_6 ) \bigr].
\end{align*}
The lengthy results for the remaining maps are available at
\cite{CrassWeb}.

\subsection{A fixed point property}

For a \G{8!}-equivariant $f$ and a point $a$ that an element $T
\in \G{8!}$ fixes,
$$T\,f(a) = f(Ta) = f(a).$$
Hence, equivariants preserve fixed points of a group element.

Being pointwise fixed by the involution
$$x_i \longleftrightarrow x_j,$$
a $28$-hyperplane
$$\LL{5}{28_{ij}} = \{x_i - x_j =0\}$$
either maps to itself or collapses to its companion $28$-point
$$
\p{28}{ij} = [\dots 0 \dots,\overbrace{1}^{i},\dots 0 \dots,
 \overbrace{-1}^{j},\dots 0 \dots]
\notin \LL{5}{28_{ij}}.
$$
In the former generic case, the map preserves the planes and lines
that are intersections of $28$-hyperplanes.  (See
Tables~\ref{tab:planes} and \ref{tab:lines}.)

\subsection{Families of equivariants}

The \G{8!} equivariants form a degree-graded module over the
\G{8!} invariants. This means that for an invariant $F_\ell$ and
equivariant $g_m$ of degrees $\ell$ and $m$, the product
$$F_\ell \cdot g_m$$
is an equivariant of degree $\ell + m$. When looking for a map in
a certain degree $k$ with special geometric or dynamical
properties, my approach is to express the entire family of
``$k$-maps" and by manipulation of parameters, locate a subfamily
with the desired behavior.

\subsection{A special map in degree four} \label{sec:4map}

In the configuration of $28$-lines \LL{1}{28_{ij}} each $8$-point lies at
the intersection of seven lines while each \LL{1}{28_{ij}} contains the
points \p{8}{i} and \p{8}{j}. (See Section~\ref{sec:specOrbs}.) Moreover,
these are the only intersections of $28$-lines. We can attempt to exploit
this structure by looking for a map with superattracting ``pipes" along
the $28$-lines: this means that, at each point on the line, the map is
critical in all five ``off-line" directions. Under such a map, the
$8$-points would be superattracting in all directions.  In addition, we
want the map's degree to be as small as possible.  Degrees two and three
avail us of too few parameters.

The family of $4$-maps has (homogeneous) dimension three:
$$
\alpha_1\,\Phi_3\,\phi_1 + \alpha_2\,\Phi_2\,\phi_2 +
\alpha_3\,\phi_4.
$$
Of course, choosing two parameter values determines a map on projective
space. Obtaining a map $g_4$ for which the $28$-lines are critical in the
off-line directions requires two parameters.  With the third parameter, we
choose a lift of the map to
\CC{7} that fixes the $8$-points. The result is
\begin{equation} \label{eq:g4}
g_4 = - 2\,\Phi_3\,\phi_1 - 9\,\Phi_2\,\phi_2  + 84\,\phi_4.
\end{equation}

The central dynamical role played by the $8$-points suggests that
a good choice of coordinates for this map places these points at
$$[1,0,0,0,0,0,0],\ \dots\ ,[0,0,0,0,0,0,1],[1,1,1,1,1,1,1].$$
Using $[v_1,\dots,v_7]$ for this system, the map takes the form
$$g_4(v) = [v_1\,T_1(v),\dots,v_7\,T_7(v)]$$
where
$$
T_k(v) = 7\,v_k^3 -4\,v_k^2\,S_1(\widehat{v_k}) +
2\,v_k\,S_2(\widehat{v_k}) - S_3(\widehat{v_k}),
$$
$S_k$ is the degree-$k$ elementary symmetric function in six variables,
and
$$\widehat{v_k} = (\dots,v_{k-1},v_{k+1},\dots).$$

In the $v$-coordinates, the equations
$$v_i=x,\ v_j=y,\ v_k=0\qquad \text{for}\ k \neq  i,j$$
determine $21$ of the $28$-lines \LL{1}{28_{ij}}. For all $k$, each term
in $S_2(\widehat{v_k})$ and $S_3(\widehat{v_k})$ contains at least one
$v_\ell\ (\ell \neq i,j)$.  Thus,
$$
S_2(\widehat{v_k})|_{\LL{1}{28_{ij}}} = 0 \quad \text{and} \quad
S_3(\widehat{v_k})|_{\LL{1}{28_{ij}}} = 0
$$
so that in the inhomogeneous coordinate $z=\tfrac{x}{y}$ on the line the
map restricts to
$$z \longrightarrow -z^3\,\frac{7\,z-4}{4\,z-7}$$
while the pair of $8$-points appear at $0$ and $\infty$.

\begin{prop}

The $8$-points \p{8}{i}, \p{8}{j} are the attractor for the restriction $g
= g_4|_{\LL{1}{28_{ij}}}$. Furthermore, in the coordinates used above, the
Julia set $J_g$ is the unit circle.

\end{prop}

\begin{proof}

Since the mobius transformation
$$\frac{7\,z-4}{4\,z-7}$$
preserves the unit disk,
$$|g(z)| < |z|^3$$
for $|z|<1$.  If $|z|<1$, iteration yields
$$
|g^n(z)|<|g^{n-1}(z)|^3< \dots < |g(z)|^{3^{n-1}}<|z|^{3^n}.
$$
Thus, every point in the disk belongs to the basin of the superattracting
fixed point $0$. Since $g$ is symmetric under
$$z \longrightarrow \frac{1}{z},$$
the basin of $\infty$ is $\{|z|>1\}$.

The complete invariance of $\{|z|=1\}$ implies that it contains
and, indeed, is the map's Julia set.

\end{proof}

Thus, the basins of the $8$-points contain the $28$-lines excepting the
\RP{1} equator---the unit circle in the coordinates above.  Along this
circle $C$, $g_4$ has periodic, preperiodic, and what we might call
chaotic saddle points.  Attached to each point $z$ on $C$ is a
$5$-dimensional stable manifold $W_z$ consisting of points attracted to
$C$, that is, to the trajectory of $z$.  Locally, the stable manifolds are
mutually disjoint and, collectively over the circle, give a stable
manifold $W_C$ of $C$ whose real-dimension is eleven and that belongs to
the basin boundaries of the pair of $8$-points.  We can see (real)
$2$-dimensional slices of this stable manifold by plotting basins of
attraction on spaces that are $g_4$-invariant and intersect $W_C$. (See
\app{app:basins}.)

By construction, $g_4$ self-maps each \Sym{6}-symmetric $28$-hyperplane
\LL{5}{28_{ij}} and hence, preserves each \CP{1} and \CP{2} intersection
of hyperplanes.  Denote these $1$- and $2$-dimensional spaces by \LL{1}{m}
and \LL{2}{n}. Furthermore, $g_4$ is equivariant under the antiholomorphic
transformation
$$x \longrightarrow \bar{x}$$
and, thereby, preserves\RR---the \Sym{8}-symmetric \RP{6}. We can get a
picture of the map's \emph{restricted dynamics} by plotting basins of
attraction on
\LL{1}{m}---a \CP{1}---and on the \RP{2} intersections
$$\LL{2}{n} \cap \RR.$$
The basin portraits appear in \app{app:basins}. Graphical and
experimental evidence support a claim of reliability for $g_4$.

\begin{conj} \label{conj:g4}

The attractor for $g_4$ is the $8$-point orbit.

\end{conj}

\begin{conj}

Under $g_4$, the basins of the $8$-points fill up \CP{6} in
measure.

\end{conj}

Finally, since $g_4$ has real coefficients (see \eqref{eq:g4}), it
preserves the \RP{6} whose points have real $u$ coordinates. This is
\emph{not} the \Sym{8}-symmetric \RR. Rather it has the \Sym{7} symmetry
of \p{8}{1} which is $[1,1,1,1,1,1,1]$ in the $u$ space.  Accordingly,
there are eight spaces of this type.

\section{Solving the octic} \label{sec:octic}

To compute a root of a polynomial, one must overcome its symmetry.
For a general equation of degree $n$ the obstacle is \Sym{n}.
Klein described a means to this end: given values for an
independent set of \Sym{n}-invariant homogeneous polynomials
$$a_1=G_1(x)\ \dots\ a_m=G_m(x),$$
find the \Sym{n} orbits of solutions $x$ to these equations.
\cite[pp. 69ff]{Klein} This task of inverting the $G_k$ is the
\emph{form problem on \Sym{n}}. It also has an inhomogeneous
manifestation: for $m-1$ given values, invert $m-1$ invariant
rational functions of degree zero.

By iterating a reliable \Sym{n} equivariant we can break the obstructing
symmetry. In effect, the dynamics provides a mechanism for solving the
form problem and, hence, the $n$th degree equation.

\subsection{Parameters}

The \G{8!} rational form problem is to solve
\begin{align*}
K_1 &= \frac{\Phi_3(u)^2}{\Phi_2(u)^3} \quad
&K_2 &=\frac{\Phi_4(u)}{\Phi_2(u)^2} \quad
&K_3 &=\frac{\Phi_5(u)}{\Phi_2(u)\,\Phi_3(u)}\\
K_4 &= \frac{\Phi_6(u)}{\Phi_2(u)^3} \quad
&K_5 &= \frac{\Phi_7(u)}{\Phi_2(u)\,\Phi_5(u)} \quad
&K_6 &= \frac{\Phi_8(u)}{\Phi_2(u)^4}.
\end{align*}
As functions, the $K_i$ define the \G{8!} quotient map
\begin{align*}
&[K_1,K_2,K_3,K_4,K_5,K_6,1] =\\
&\qquad [\Phi_2\,\Phi_3^3\,\Phi_5,
 \Phi_2^2\,\Phi_3\,\Phi_4\,\Phi_5,
 \Phi_2^3\,\Phi_5^2,
 \Phi_2\,\Phi_3\,\Phi_5\,\Phi_6,
 \Phi_2^3\,\Phi_3\,\Phi_7,
 \Phi_3\,\Phi_5\,\Phi_8,
 \Phi_2^4\,\Phi_3\,\Phi_5]
\end{align*}
on $\CP{6} \setminus \{\Phi_2=\Phi_3=\Phi_5=0\}$. The generic
fiber over points in \CP{6} is a \G{8!} orbit given by
\begin{align*}
&\{\Phi_3^2 = K_1\,\Phi_2^3\} \cap \{\Phi_4 = K_2\,\Phi_2^2\} \cap
\{\Phi_5 = K_3\,\Phi_2\,\Phi_3\}\\
&\cap \{\Phi_6 = K_4\,\Phi_2^3\} \cap \{\Phi_7 =
K_5\,\Phi_2\,\Phi_5\} \cap \{\Phi_8 = K_6\,\Phi_2^4\}.
\end{align*}
Exceptional locations are
$$[0,0,1,0,0,0,0],\ [0,0,0,0,1,0,0],\ [0,0,0,0,0,1,0]$$
where the respective fibers are the hypersurfaces $\{\Phi_3=0\}$,
$\{\Phi_5=0\}$, and $\{\Phi_2=0\}$. The parameters $K_i$ forge a link
between octic equations and \G{8!} actions. The connection consists in
$K$-parametrizations of each regime.

\subsection{A family of \Sym{8} quintics}

Let \G{v} be a version of \G{8!} that acts on a $v$-coordinatized
$\CP{6}_v$. This will be a parameter space---the coordinate $v$
merely stands in for $u$. The linear polynomials
\begin{equation} \label{eq:X_k}
X_k(x) = - 7\,x_k + \sum_{i \neq k} x_i
\end{equation}
form an orbit of size eight.  Let
$$L_k(u)=X_k(\obar{H^T}u).$$
Then the rational functions
$$\sigma_k(v) = \frac{\Phi_2(v)\,L_{k}(v)}{\Phi_3(v)} $$
also give an $8$-orbit. Taking the $\sigma_k$ as roots of a polynomial
$$
R_v(s) = \prod_{k=1}^8 \bigl(s - \sigma_k(v)\bigr) = \sum_{k=0}^8
C_k(v)\,s^{8-k}
$$
creates a family of octics whose members generically have \Sym{8}
symmetry. Since \G{v} permutes the $\sigma_k(v)$, each coefficient
$C_k(v)$ is \G{v}-invariant and hence, expressible in the basic forms
$\Phi_k(v)$. Ultimately, we can express each $C_k$ in terms of the $K_i$.
By direct calculation,
\begin{align*}
C_0 &= 1\\
C_1 &= 0\\
C_2 &= -\frac{32}{K_1}\\
C_3 &= \frac{512}{3\,K_1^2}\\
C_4 &=  \frac{512\,\bigl(1 - 2\,K_2\bigr)}{8\,K_1^2}\\
C_5 &= \frac{16384\,( -5 + 6\,K_3)}{15\,K_1^2}\\
C_6 &= \frac{-16384\,\bigl(- 3 + 8\,K_1 + 18\,K_2 - 24\,K_4 \bigr)}{9\,K_1^3}\\
C_7 &= \frac{262144\,(35 - 70\,K_2 - 84\,K_3 - 120\,K_3\,K_5) }{105\,K_1^3}\\
C_8 &= \frac{131072\,\bigl(15 - 160\,K_1 + 180\,K_2 + 180\,K_2^2 +
384\,K_1\,K_3 + 480\,K_4 - 720\,K_6 \bigr)}{45\,K_1^4}.
\end{align*}
Members of the family of octic \emph{resolvents}
$$R_K(s) = \sum_{n=0}^8 C_n\,s^{8-n}$$
parametrized by $K=(K_1,\dots,K_6)$ are particularly well-suited
for an iterative solution that employs $g_4$. For chosen values of
the $K_i$, a solution to the resulting form problem yields a root
of $R_K$. We can use \G{8!} symmetry to find such a solution
without explicitly inverting the $K_i$ equations.  Our attention
will turn to this issue after we connect the general octic to the
special family $R_K$.

\subsection{Reduction of the general octic to the \G{8!} resolvent}

By means of a linear \emph{Tschirnhaus} transformation
$$x \longrightarrow y - \frac{a_1}{8}$$
the general octic
$$p(x) = x^8 + a_1\,x^7 + a_2\,x^6 + a_3\,x^5 + a_4\,x^4 + a_5\,x^3
+ a_6\,x^2 + a_7\,x + a_8$$ becomes the \emph{standard}
$7$-parameter \emph{resolvent}
$$
q(y) = y^8 + b_2\,y^6 + b_3\,y^5 + b_4\,y^4 + b_5\,y^3 + b_6\,y^2
+ b_7\,y + b_8
$$
where \small \begin{align*}
b_2 =&\ \frac{-7\,a_1^2 + 16\,a_2}{16}\\
b_3 =&\ \frac{7\,a_1^3 - 24\,a_1\,a_2 + 32\,a_3}{32}\\
b_4 =&\ \frac{-105\,a_1^4 + 480\,a_1^2\,a_2 - 1280\,a_1\,a_3 +
     2048\,a_4}{2048}\\
b_5 =&\ \frac{7\,a_1^5 - 40\,a_1^3\,a_2 + 160\,a_1^2\,a_3 -
     512\,a_1\,a_4 + 1024\,a_5}{1024}\\
b_6 =&\ \frac{-35\,a_1^6 + 240\,a_1^4\,a_2 -
     1280\,a_1^3\,a_3 + 6144\,a_1^2\,a_4 -
     24576\,a_1\,a_5 + 65536\,a_6}{65536}\\
b_7 =&\ \frac{3\,a_1^7 - 24\,a_1^5\,a_2 + 160\,a_1^4\,a_3 -
     1024\,a_1^3\,a_4 + 6144\,a_1^2\,a_5 -
     32768\,a_1\,a_6 + 131072\,a_7}{131072}\\
b_8 =&\ \frac{-7}{16777216}\,\bigl( a_1^8 + 64\,a_1^6\,a_2
     - 512\,a_1^5\,a_3 + 4096\,a_1^4\,a_4
     - 32768\,a_1^3\,a_5 + 262144\,a_1^2\,a_6\\
     &- 2097152\,a_1\,a_7 + 16777216\,a_8 \bigr).
\end{align*} \normalsize

Application of another linear \emph{Tschirnhaus} transformation
$$s \longrightarrow \frac{y}{\lambda}$$
converts the $6$-parameter family $R_K(s)$ into a \G{8!} resolvent
\begin{align*}
\Sigma_{K,\lambda}(y) &=
\lambda^8\,R_K\Bigl(\frac{y}{\lambda}\Bigr) =\sum_{n=0}^8
\lambda^n\,C_n\,y^{8-n}
\end{align*}
in the seven parameters $K_1,\dots,K_6$, and the \emph{auxiliary}
$\lambda$.

The functions
$$b_n = \lambda^n\,C_n(K)$$
relate the coefficients of $q$ and $\Sigma_{K,\lambda}$. The $b_n$
invert to
\begin{align*}
K_1 &= \frac{-9\,b_3^2}{8\,b_2^3} \\
K_2 &= \frac{b_2^2 - 2\,b_4}{2\,b_2^2} \\
K_3 &= \frac{5\,\bigl( b_2\,b_3 - b_5 \bigr) }{6\,b_2\,b_3} \\
K_4 &= \frac{2\,b_2^3 - 3\,b_3^2 - 6\,b_2\,b_4 + 6\,b_6}{8\,b_2^3}\\
K_5 &= \frac{7\, \bigl( b_2^2\,b_3 - b_3\,b_4 - b_2\,b_5 + b_7
\bigr)}
{10\,b_2\,\bigl( b_2\,b_3 - b_5 \bigr) } \\
K_6 &= \frac{b_2^4 - 4\,b_2\,b_3^2 - 4\,b_2^2\,b_4 + 2\,b_4^2 +
4\,b_3\,b_5 + 4\,b_2\,b_6 - 4\,b_8}{8\,b_2^4} \\
\lambda &= \frac{-3\,b_3}{16\,b_2}.
\end{align*}
Thus, almost any octic descends to a member of $R_K$. The
reduction fails when
$$
-7\,a_1^2 + 16\,a_2 = 16\,b_2 = 0 \qquad \text{or} \qquad 7\,a_1^3
- 24\,a_1\,a_2 + 32\,a_3 = 32\,b_3 = 0
$$
or
\begin{align*}
&-105\,a_1^5 + 600\,a_1^3\,a_2 - 768\,a_1\,a_2^2 - 608\,a_1^2\,a_3
+
1024\,a_2\,a_3 + 512\,a_1\,a_4 - 1024\,a_5 \\
&\qquad = 1024\,(b_2\,b_3 - b_5) = 0.
\end{align*}
A solution to the special resolvent $R_K$ then ascends to a
solution to the general quintic.

\subsection{A family of \Sym{8} actions}

With the basic \G{v} forms and maps, we can define the
\emph{parametrized change of coordinates}
$$
u = \tau_v w = \sum_{k=1}^7 \bigl(\Phi_{9-k}(v)\,\phi_k(v) \bigr)
w_k.
$$
For a choice of parameter $v$,
$$\tau_v: \CP{6}_w \longrightarrow \CP{6}_u$$
is linear in $w$ and provides a parametrized family of \G{8!} groups
$$\G{w}^v = \tau_v^{-1} \G{u} \tau_v.$$
A matrix form results from taking the $\phi_k(v)$ as column
vectors:
$$
\begin{pmatrix} u_1\\ \vdots\\u_7 \end{pmatrix} = \begin{pmatrix}
&&\\
\Phi_8(v)\,\phi_1(v)&\dots&\Phi_2(v)\,\phi_7(v)\\
&&
\end{pmatrix}
\begin{pmatrix} w_1\\ \vdots\\w_7 \end{pmatrix}.
$$
The setup here is as follows.
\begin{itemize}

\item \G{u} is a version of \G{8!} that acts on a \emph{reference space}
$\CP{6}_u$.

\item \G{v} is a version of \G{8!} that acts on a \emph{parameter space}
$\CP{6}_v$.

\item \G{u} and \G{v} have identical expressions in their respective
coordinates.

\item $\G{w}^v$ are versions of \G{8!} that act on
\emph{v-parametrized spaces} $\CP{6}_w$.

\item The iteration that solves octics $R_K$ takes place in $\CP{6}_w$.

\end{itemize}

Each $\G{w}^v$ has its system of invariants and equivariants. From
this point of view we can see, in the resolvents $R_v$ and
$\G{w}^v$ equivariants, a connection between octics and dynamical
systems. Furthermore, each $\G{w}^v$ invariant and equivariant is
expressible in the $K_i$.  This circumstance connects a resolvent
$R_K$ with $\G{w}^v$-symmetric maps.

By construction, $\tau_v w$ possesses an equivariance property:
$$
\tau_{A v} w = A\,\tau_v w \quad \text{for}\ A \in \G{v}, \G{u}.
$$
The determinant of $\tau_v$ will enter into upcoming calculations
and so, demands some attention. Since
$$\det{\tau_{A v}} = \det{A}\,\det{\tau_v},$$
$\det{\tau_v}$ is invariant under the \Alt{8} subgroup \G{8!/2} of
\G{v} but only relatively invariant under the full \Sym{8} group
\G{8!}. (The ``even transformations" have determinant $1$ while
the odd elements have determinant $-1$.) Furthermore,
\begin{align*}
\det{\tau_v} &= \Bigl(\prod_{k=2}^8  \Phi_k(v) \Bigr)\,
\det{\begin{pmatrix}\phi_1(v)\  \cdots\  \phi_7(v)\end{pmatrix}}\\
&= \Bigl(\prod_{k=2}^8 \Phi_k(v) \Bigr)\,\Psi_{28}(v)
\end{align*}
where $\Psi_{28}$ is, according to a basic result in reflection group
theory, a scalar multiple of the product of the $28$ linear forms
associated with the $28$ hyperplanes that are fixed by the reflections
that generate
\G{v}. Furthermore, $\Psi_{28}$ is invariant under the group
\G{8!/2} (isomorphic to the alternating group
\Alt{8}) but is relatively invariant under \G{8!}. Consequently,
the degree-$126$ square of $\det{\tau_v}$ is \G{8!}-invariant with
$K$-expression
$$(\det{\tau_v})^2 = \Phi_2^{63}(v)\,t_K.$$
The explicit form of $t_K$ appears at \cite{CrassWeb}.

\subsection{A family of \Sym{8} invariants}

The equivariance in $v$ of $\tau_v w$ implies that $\Phi_2(\tau_v
w)$ is \G{v}-invariant. Thus, each $w$ coefficient of
$\Phi_2(\tau_v w)$ inherits the same invariance. Since
$$
\deg_v \Phi_2(\tau_v w) = \deg_u \Phi_2(u) \cdot \deg_v \tau_v w =
2 \cdot 9 = 18,
$$
the rational function
$$\frac{\Phi_2(u)}{\Phi_2(v)^9} = \frac{\Phi_2(\tau_v w)}{\Phi_2(v)^9}$$
is degree zero in $v$ and thereby, expressible in $K$. For each
$w$ monomial, we can solve a system of linear equations whose
dimension is that of the degree-18 \G{v} invariants.  The result
is an explicit expression in $K$ for each $w$-coefficient of
$\Phi_2(\tau_v w)$. Let
\begin{equation} \label{eq:Phi2}
\Phi_2(v)^9\,\Phi_{2_K}(w) = \Phi_2(u)
\end{equation}
define the basic degree-2 $\G{w}^v$ invariant $\Phi_{2_K}(w)$.
Similar considerations apply in degree three where
\begin{equation} \label{eq:Phi3}
\Phi_2(v)^{12}\,\Phi_3(v)\,\Phi_{3_K}(w) = \Phi_3(u).
\end{equation}
The results appear at \cite{CrassWeb}.

By Proposition~\ref{prop:PhiInG}, the degree-$4$ and degree-$5$
invariants derive from those in degrees two and three. The chain
rule determines a transformation formula for the bordered hessian.
Let $|\cdot |$ represent the determinant and $A^T$ the transpose
of $A$.

\begin{prop}

For $y=A x$,
$$
BH_x\bigl(F(y),G(y),J(y)\bigr) =
\begin{pmatrix}A^T&0\\0&1 \end{pmatrix} BH_y\bigl(F(y),G(y),J(y)\bigr)
\begin{pmatrix}A&0\\0&1 \end{pmatrix}
$$
where the subscript indicates the variable of differentiation.
Thus,
$$
\bigl|BH_x\bigl(F(y),G(y),J(y)\bigr) \bigr| =
 |A|^2 \bigl|BH_y\bigl(F(y),G(y),J(y)\bigr) \bigr|.
$$

\end{prop} \noindent
Applied to the parametrized change of variable $w = \tau_v^{-1}
u$, this formula along with \eqref{eq:Phi2} and \eqref{eq:Phi3}
yields
\begin{align*}
G_4(u) &= \bigl|BH_u \bigl(\Phi_2(u),\Phi_3(u),\Phi_3(u)\bigr) \bigr|\\
&= \Bigl|BH_u \bigl( \Phi_2(v)^9\,\Phi_{2_K}(w),
\Phi_2(v)^{12}\,\Phi_3(v)\,\Phi_{3_K}(w),
\Phi_2(v)^{12}\,\Phi_3(v)\,\Phi_{3_K}(w)
\bigr) \Bigr|\\
&= \frac{\Phi_2(v)^{6\cdot 9 + 2\cdot
12}\,\Phi_3(v)^2}{|\tau_v|^2}\,
\Bigl|BH_w \bigl(\Phi_{2_K}(w),\Phi_{3_K}(w),\Phi_{3_K}(w) \bigr) \Bigr|\\
&= \frac{\Phi_2(v)^{78}\,\Phi_3(v)^2}{\Phi_2(v)^{63}\,t_K}\,
\Bigl|BH_w \bigl(\Phi_{2_K}(w),\Phi_{3_K}(w),\Phi_{3_K}(w) \bigr) \Bigr| \\
&= \frac{\Phi_2(v)^{15}\,\Phi_3(v)^2}{t_K}\,
\Bigl|BH_w \bigl(\Phi_{2_K}(w),\Phi_{3_K}(w),\Phi_{3_K}(w) \bigr) \Bigr| \\
&= \frac{\Phi_2(v)^{18}}{t_K}\,\frac{\Phi_3(v)^2}{\Phi_2(v)^3}
\Bigl|BH_w \bigl(\Phi_{2_K}(w),\Phi_{3_K}(w),\Phi_{3_K}(w) \bigr) \Bigr| \\
&= \Phi_2(v)^{18}\,\frac{K_1}{t_K}
\Bigl|BH_w \bigl(\Phi_{2_K}(w),\Phi_{3_K}(w),\Phi_{3_K}(w) \bigr) \Bigr| \\
&= \Phi_2(v)^{18}\,G_{4_K}(w).
\end{align*}
Using Proposition~\ref{prop:PhiInG} we obtain
\begin{align*}
\Phi_4(u) &= \frac{1}{576}\,\bigl(
72\,\Phi_2(v)^{18}\,\Phi_{2_K}(w)^2 + \Phi_2(v)^{18}\,G_{4_K}(w)
\bigr)\\
&= \Phi_2(v)^{18}\,\Phi_{4_K}(w)
\end{align*}
so that
\begin{align*}
G_5(u) &= \bigl|BH_u \bigl(\Phi_2(u),\Phi_3(u),\Phi_4(u)\bigr) \bigr|\\
&= \Bigl|BH_u \bigl( \Phi_2(v)^9\,\Phi_{2_K}(w),
\Phi_2(v)^{12}\,\Phi_3(v)\,\Phi_{3_K}(w),
\Phi_2(v)^{18}\,\Phi_{4_K}(w)
\bigr) \Bigr|\\
&= \Phi_2(v)^{21}\,\Phi_3(v)\,G_{5_K}(w)
\end{align*}
and
\begin{align*}
\Phi_5(u) &= \frac{1}{768}\,\bigl(
96\,\Phi_2(v)^9\,\Phi_{2_K}(w)\,\Phi_2(v)^{12}\,\Phi_3(v)\,\Phi_{3_K}(w)
+ \Phi_2(v)^{21}\,\Phi_3(v)\,G_{5_K}(w)
\bigr)\\
&= \Phi_2(v)^{21}\,\Phi_3(v)\,\Phi_{5_K}(w).
\end{align*}
Employed here are the obvious definitions
\begin{align*}
G_{4_K}(w) &= \frac{K_1}{t_K} \Bigl|BH_w
\bigl(\Phi_{2_K}(w),\Phi_{3_K}(w),\Phi_{3_K}(w) \bigr)
\Bigr| \\
G_{5_K}(w) &= \frac{1}{t_K} \Bigl|BH_w
\bigl(\Phi_{2_K}(w),\Phi_{3_K}(w),\Phi_{4_K}(w) \bigr) \Bigr|
\end{align*}
as well as natural definitions for $\Phi_{4_K}(w)$ and
$\Phi_{5_K}(w)$.

\subsection{A family of \Sym{8} equivariants} \label{sec:S8equivFam}

Emerging from each $\G{w}^v$ action is a version
$\tau_v^{-1}\,g_4(\tau_v w)$ of $g_4(u)$. Being \G{v}-invariant,
these maps also admit parametrization by $K$. In this way, each
octic $R_K$ enters into association with a dynamical system $g_K$
on $\CP{6}_w$.

The reversed identity $R$ and gradient $\nabla^r=R\,\nabla$
appeared in the context of a change from eight $x$ coordinates to
seven $u$ coordinates. (See Definition~\ref{df:R}.) In the present
setting, a \emph{reversed transpose} arises.

\begin{defn}

The \emph{repose} $A^r$ of an $n \times n$ matrix $A$ is its
reflection through the \emph{reversed diagonal}---the entries
whose subscripts sum to $n+1$. Alternatively,
$$A^r = R\, A^T\,R.$$

\end{defn}

\begin{prop}

For a coordinate change $x=Ay$ and a function $F(y) =
\tilde{F}(x)$, the reversed gradient map transforms by
$$\nabla_y^r F(y) = A^r\,\nabla_x^r \tilde{F}(x).$$

\end{prop}

\begin{proof}

See \cite{quintic}, Proposition~4.2, but note that the coordinate change
there should be $w=Au$.

\end{proof}
Using \eqref{eq:equiv}, the degree-$1$ \G{8!} map is
\begin{align*}
\phi_1(u) &= -\frac{8}{2}\nabla_u^r \Phi_2(v)^9\,\Phi_{2_K}(w) \\
&= -4\,\Phi_{2}(v)^9\,
 (\tau_v^{-1})^r\,\nabla_w^r \Phi_{2_K}(w) \\
&= -4\,\Phi_{2}(v)^9\,\tau_v\,\tau_v^{-1}\,(\tau_v^{-1})^r\,
   \nabla_w^r \Phi_{2_K}(w) \\
&= -4\,\tau_v\,\Phi_{2}(v)^9\,(\tau_v^r\,\tau_v)^{-1}\,
   \nabla_w^r \Phi_{2_K}(w).
\end{align*}
Thus,
$$
\tau_v^{-1}\,\phi_1(\tau_v w) = -4\,\Phi_{2}(v)^9\,
(\tau_v^r\,\tau_v)^{-1}\,\nabla_w^r \Phi_{3_K}(w).
$$
Using the description on the left-hand side, a straightforward
calculation reveals this map to be invariant in $v$ so that the
matrix $\tau_v^r\,\tau_v$ has entries that are degree-$18$ \G{v}
invariants. Hence, the matrix product has the $K$-expression
\begin{equation} \label{eq:TK}
\tau_v^r\,\tau_v = \Phi_2(v)^9\,T_K \quad \text{or} \quad
(\tau_v^r\,\tau_v)^{-1} = \frac{T_K^{-1}}{\Phi_2(v)^9}.
\end{equation}
(See \cite{CrassWeb} for the explicit form.) Also, note that
$$
\det{T_K} = \frac{\det{(\tau_v^r\,\tau_v)}}{\Phi_2(v)^{9\cdot 7}}
= \frac{(\det{\tau_v})^2}{\Phi_2(v)^{63}} = t_K.
$$
Making use of \eqref{eq:TK} to express the transformation of basic
equivariants yields
\begin{align*}
\phi_1(u) &= -4\,\Phi_{2}(v)^9\,\tau_v\,
  \frac{T_K^{-1}}{\Phi_2(v)^9}\,
   \nabla_w^r \Phi_{2_K}(w) \\
&= -4\,\tau_v\,T_K^{-1}\,\nabla_w^r \Phi_{2_K}(w).
\end{align*}

As for the other relevant maps, similar calculations give
\begin{align*}
\phi_2(u) &= -\frac{8}{3}\,\Phi_{2}(v)^{12}\,\Phi_3(v)\,\tau_v\,
  \frac{T_K^{-1}}{\Phi_2(v)^9}\,
   \nabla_w^r \Phi_{3_K}(w) \\
&= -\frac{8}{3}\,\Phi_{2}(v)^{3}\,\Phi_3(v)\,\tau_v\,T_K^{-1}\,
   \nabla_w^r \Phi_{3_K}(w)\\
\phi_4(u) &= -\frac{8}{4}\,\Phi_{2}(v)^{21}\,\Phi_3(v)\,\tau_v\,
  \frac{T_K^{-1}}{\Phi_2(v)^9}\,
   \nabla_w^r \Phi_{5_K}(w) \\
&= -2\,\Phi_{2}(v)^{12}\,\Phi_3(v)\,\tau_v\,T_K^{-1}\,
 \nabla_w^r \Phi_{5_K}(w).
\end{align*}

Finally, we can identify a $K$-parametrized $4$-map $g_K(w)$ that
is conjugate to $g_4(u)$. The map's expression in basic terms
appears after substitution into the formula found in
Section~\ref{sec:4map}:
\begin{align*}
g_4(u) =&\ 84\,\phi_4(u) - 9\,\Phi_2(u)\,\phi_2(u) - 2\,\Phi_3(u)\,\phi_1(u)\\
=&\ -8\,\Phi_2^{12}(v)\,\Phi_3(v)\,\tau_v\,T_K^{-1}\,\Bigl(
21\,\nabla_w^r\Phi_{5_K}(w) \\
&\ - 3\,\Phi_{2_K}(w)\,\nabla_w^r\Phi_{3_K}(w) -
\Phi_{3_K}(w)\,\nabla_w^r\Phi_{2_K}(w) \Bigr).
\end{align*}
Thus, we have a $K$-parametrized family of $4$-maps on $\CP{6}_w$:
$$
g_K (w) = T_K^{-1}\,\Bigl( 21\,\nabla_w^r\Phi_{5_K}(w) -
3\,\Phi_{2_K}(w)\,\nabla_w^r\Phi_{3_K}(w) -
\Phi_{3_K}(w)\,\nabla_w^r\Phi_{2_K}(w) \Bigr)
$$
whose relation to the reference $4$-map is
$$g_4(u)= -8\,\Phi_2^{12}(v)\,\Phi_3(v)\,\tau_v\,g_K(w).$$

\subsection{Root selection}

Being conjugate to $g_4(u)$ each $g_K(w)$ shares the former's
conjectured reliable dynamics. Accordingly, the attractor for each
choice of $K$ is the $8$-point orbit in the corresponding
$\CP{6}_w$ so that for almost every $w_0 \in \CP{6}_w$,
$$
g_K^n(w_0) \longrightarrow \tau_v^{-1} \p{8}{\ell} \qquad
\text{for some $8$-point}\ \p{8}{\ell} \in \CP{6}_u.
$$
To solve the resolvent $R_K$, the output of the iteration must
link with the roots of $R_K$. From here, we see that solving $R_K$
amounts to inverting $\tau_v$---the form problem in another guise.
This is effectively what the dynamics of $g_K$ accomplishes with
the assistance of a \G{8!} tool that I now describe.

The quadratic \Sym{7}-invariants
$$X^2_k(x) = - 7\,x_k^2 + \sum_{i \neq k} x_i^2 $$
form a \G{8!} orbit of size eight. Recall that
$$L_k(u) = X_k(\obar{H^T} u)$$
and let
$$Q_k(u) = X^2_k(\obar{H^T} u)$$
be the $\HH_u$ expression for $X_k^2$. Furthermore, each of the
eight forms
$$
G_k(u) = Q_k(u) - \frac{3}{4}\,L_k(u)^2 \quad
  k = 1, \dots, 8
$$
vanish at the $8$-points \p{8}{\ell} with $\ell \neq k$ but not at
\p{8}{k}.

Now, consider the rational function
$$J_v(w) =
\alpha\,\sum_{k=1}^8
 \frac{G_k(\tau_v w)}{\Phi_2(\tau_v w)}\,\frac{\Phi_2(v)\,L_k(v)}{\Phi_3(v)}
= \alpha\,\sum_{k=1}^8 \frac{G_k(\tau_v w)}{\Phi_2(\tau_v w)}\,\sigma_k(v)
$$
where $\alpha$ is a constant to be determined. Since the
$v$-degree of the numerator and denominator is $21=2\cdot 9 + 3$
while the $w$-degree is $2$, the function is rationally degree
zero in both variables. At an $8$-point $\tau_v^{-1} \p{8}{\ell}$
in $\CP{6}_w$ seven of the eight terms in $J_v$ vanish; this
leaves
$$\alpha\,\frac{G_\ell(\p{8}{\ell})}{\Phi_2(\p{8}{\ell})}\,\sigma_\ell(v).$$
Choosing
$$
\alpha = \frac{\Phi_2(\p{8}{k})}{G_k(\p{8}{k})} = \frac{1}{48}\qquad
(k=1,\dots,8)
$$
``selects" the root $\sigma_\ell(v)$ of $R_K(s)$. Since the iterative
``output" of $g_K(w)$ is a single $8$-point in $\CP{6}_w$, the dynamics
produces one root.

To obtain a usable form of the root-selector $J_v(w)$, let
$$\Gamma_v(w) = \sum_{k=1}^8 G_k(\tau_v w)\,L_k(v).$$
Since \G{v} permutes its terms, $\Gamma_v$ is invariant under the
action and hence, expressible in $K$:
$$\Gamma_v(w) = \Phi_2(v)^8\,\Phi_3(v)\,\Gamma_K(w).$$
(The explicit form of $\Gamma_K$ appears at \cite{CrassWeb}.)
Application of \eqref{eq:Phi2} yields
\begin{align*}
J_v(w)
&= \frac{\Phi_2(v)\,\Gamma_v(w)}{48\,\Phi_3(v)\,\Phi_2(\tau_v w)}\\
J_K(w) &= \frac{\Gamma_K(w)}{48\,\Phi_{2_K}(w)}.
\end{align*}

\subsection{The procedure summarized}

At \cite{CrassWeb}, there are \emph{Mathematica} data files and a
notebook that implement the iterative solution to the octic.

\begin{enumerate}

\item Select a general $8$-parameter octic $p(x)$.

\item Tschirnhaus transform $p(x)$ into a member $R_K(s)$ of the
$6$-parameter family of \G{8!} octics---this determines values for
$K_1, \dots, K_6$ as well as the auxiliary parameter $\lambda$.

\item For the selected $K$ values, compute the matrix $T_K$, the invariants
$\Phi_{2_K}(w)$ and $\Phi_{3_K}(w)$, the $4$-map $g_K(w)$, the
form $\Gamma_K(w)$, and the root-selector $J_K(w)$.

\item From an arbitrary initial point $w_0$ iterate $g_K$ until
convergence:
$$g_K^n(w_0) \longrightarrow w_\infty.$$
Conjecturally, the output $w_\infty$ is an $8$-point in
$\CP{6}_w$.

\item Compute a root $\sigma = J_K(w_\infty)$ of $R_K$.

\item Transform $\sigma$ into a root of $p(x)$.

\end{enumerate}

\section{Beyond the octic} \label{sec:gen}

\subsection{The general case}

In the following discussion the \Sym{n} actions under
consideration derive from permutation of coordinates on the
\Sym{n}-invariant hyperplanes
$$\HH^{n-1}=\Biggl\{\sum_{k=1}^{n} x_k =0\Biggr\}.$$
These irreducible representations of \Sym{n} project to actions
\G{n!} on $\PP{}\HH^{n-1}\simeq \CP{n-2}$.

For equations of degree $n \neq 8$, does the analogue of the
octic-solving algorithm exist? Evidently, the reduction of the
$n$th degree polynomial to an $(n-2)$-parameter family of \G{n!}
resolvents is general.

\begin{query}

Is there an \Sym{n}-equivariant $4$-map on $\PP{}\HH^{n-1}$ that,
on a $\binom{n}{2}$-line \LL{1}{} (where all but two coordinates
are equal):

\begin{itemize}

\item superattracts in the off-line directions

\item restricts to a reliable map whose attractor consists of the pair of
$n$-points on \LL{1}{}?

\end{itemize}

\end{query}

\begin{query}
If so, is the map expressible as
$$
z \longrightarrow z^3\,\frac{a\,z + b}{b\,z + c} \qquad a,b
\in\R{} \quad \text{and}\ |b|<|a|
$$
when restricted to \LL{1}{}?
\end{query}

The affine space $\{x_n \neq 0\}$ parametrized by
$$\Bigl[x_1, \dots, x_{n-2},-\sum_{k=1}^{n-2} x_k -1,1\Bigr]$$
is tangent to $\PP{} \HH^{n-1}$ at the affine part of \LL{1}{}
given by
$$[\zeta, \dots, \zeta,(2-n)\,\zeta-1,1].$$
We can identify these spaces respectively with
$$
\{(x_1, \dots, x_{n-2})\}\simeq \CC{n-2}\quad \text{and}\quad
\{(\zeta, \dots, \zeta)\}\simeq\CC{}.
$$

\begin{defn}

Abusing notation, let \LL{1}{} be the affine part of an
$\binom{n}{2}$-line. Use $\T_{\LL{1}{}}$ to denote the tangent
space to $\PP{}\HH^{n-1}$ along \LL{1}{}.  Also,
$$
\LL{1}{\perp} =
  \Biggl\{\sum_{k=1}^{n-2} x_k = 0\Biggr\} \simeq \CC{n-3}
$$
is the (euclidean) orthogonal complement in $\T_{\LL{1}{}}$ to
\LL{1}{}.

\end{defn}

The subgroup $S_{\LL{1}{}}$ of \G{n!} that stabilizes $\T_{\LL{1}{}}$ is
isomorphic to \Sym{n-2} and acts by permutations on \CC{n-2}.  The action
of $S_{\LL{1}{}}$ fixes $\T_{\LL{1}{}}$ and \LL{1}{\perp} set-wise and
fixes \LL{1}{} in a point-wise manner.

\begin{rem}

In the treatment of \LL{1}{} below, we need not worry about the
point at infinity
$$[1,\dots,1,2-n,0],$$
since this point is in the same \G{n!}-orbit as the affine point
$$[1,\dots,1,0,2-n]$$
which we have identified with
$$\frac{1}{2-n}\,(1,\dots,1).$$

\end{rem}

For $n \geq 5$, the family of $4$-maps on $\PP{}\HH^{n-1}$ is, in
homogeneous parameters,
$$
g_\alpha = \alpha_1\,f_4 + \alpha_2\,F_2\,f_2 + \alpha_3\,F_3\,f_1
$$
where definitions of $F_k$ and $f_k$ are obvious extensions from the
\G{8!} case.  (For $n<5$, the family of $4$-maps is not $3$-dimensional.)
Do three parameters suffice to obtain a map $g_{\alpha}$ with jacobian
matrix $g_{\alpha}'(z)$ at $z$ such that
$$g_{\alpha}'(z)\,\LL{1}{\perp} = 0 \qquad \text{for all}\ z \in \LL{1}{}?$$

First of all, symmetry demands that, for each $g_\alpha$ and all $z \in
\LL{1}{}$, there is only one ``off-line" eigenvalue of $g_\alpha'(z)$.

\begin{lemma} \label{lm:f'}

For an action \G{} on \CP{n} and a \G{}-equivariant $f$ that is
holomorphic at $a$, the jacobian $f'(a)$ is equivariant on the
tangent space $\T_a \simeq \CC{n}$ under the stabilizer $\Ss_a$ of
$a$.

\end{lemma}

\begin{proof}

After treating several technical matters, the proof amounts to a
simple calculation.

\begin{enumerate}

\item Take $a$ to be $[0,\dots,0,1]$ and $\T_a$ to be
$\{x_{n+1} \neq 0\}$ and lift them to
$$
\hat{a}=(0,\dots,0,1)
\qquad \text{and} \qquad
\widehat{\T{a}}=\{x_{n+1}=1\}.
$$

\item Strictly speaking, $\Ss_a$ is a group of projective
transformations. Here, we choose linear representatives $T$ of $\Ss_a$
that, as maps on \CC{n+1}, satisfy $T a=a$. This group of linear
transformations acts on $\widehat{\T_a}$.

\item For a homogeneous polynomial $G(x)$ of degree $r$,
$$r\,G(x) = \bigl(\nabla_x G(x)\bigr)^T x$$
is a familiar identity.  Generalized to a rational map $g(x)$, the
result is
$$r\,g(x) = g'(x)\,x.$$

\end{enumerate}

Given $T \in \G{}$, let $\hat{T}$ be a lift to \CC{n+1}.  Then for all $x
\in \CC{n+1}$,
\begin{align*}
\bigl(f'(Tx)\,T\bigr) x &= \bigl(f'(Tx)\bigr) Tx\\
&= r\,f(Tx)\\
&= T(r\,f(x))\\
&= T(f'(x)\,x)\\
&= \bigl(T f'(x)\bigr) x.
\end{align*}
Thus,
$$f'(Tx)\,T = T f'(x).$$
In particular, when $x=a$ and $T \in \Ss_a$,
$$f'(a)\,T = f'(T a)\,T = T\,f'(a).$$

\end{proof}

\begin{lemma} \label{lm:L1Perp}

A linear $S_{\LL{1}{}}$-equivariant $T$ preserves \LL{1}{\perp}.

\end{lemma}

\begin{proof}

The transformation
$$A:\T_{\LL{1}{}} \longrightarrow \T_{\LL{1}{}}$$
that cyclically permutes coordinates according to $(12\dots
(n-3)(n-2))$ has $n-2$ eigenspaces, namely, the lines \LL{1}{k}
given by
$$
\LL{1}{k} = \Biggl\{ t\,w_k\ |\ t \in \CC{},\
  w_k = \begin{pmatrix}
  1\\ \w{}^k\\ \w{}^{2\,k}\\ \vdots\\ \w{}^{(n-3)\,k}
  \end{pmatrix},\
  \w{} = e^{2\,\pi\,i/(n-2)}
\Biggr\} \quad k=1, \dots, n-2.
$$
Note that \LL{1}{}=\LL{1}{n-2} and $\{w_k\ |\ k=1, \dots, n-3\}$
is a basis for \LL{1}{\perp}.  Moreover, each line \LL{1}{k} has
the eigenvalue $\w{}^{(n-3)\,k}$.

Since $A \in S_{\LL{1}{}}$ and
$$A (T w_k)= T (A w_k) = \w{}^{(n-3)\,k}\, T w_k,$$
$T w_k$ is also an $\w{}^{(n-3)\,k}$-eigenvector of $A$.  Thus,
each \LL{1}{k} is an eigenspace of $T$ so that \LL{1}{\perp} is
$T$-invariant.

\end{proof}

\begin{prop} \label{prop:psi'}

For all $z \in \LL{1}{}$, the eigenvectors of the jacobian
$g_\alpha'(z)$ span $\LL{1}{\perp}$. Moreover, all associated
eigenvalues are equal, making $\LL{1}{\perp}$ an eigenspace of
$g_\alpha'(z)$.

\end{prop}

\begin{proof}

Recall that the stabilizer $\Sym{\LL{1}{}}$ is isomorphic to
\Sym{n-2}. (The $n-2$ things that $\Sym{\LL{1}{}}$ permutes are
the vectors $w_k$ defined above.) By Lemmas~\ref{lm:f'} and
\ref{lm:L1Perp}, let $v \in \LL{1}{\perp}$ be an eigenvector of
$g_\alpha'(z)$ with eigenvalue $\lambda$. For $A \in
S_{\LL{1}{}}$, Lemma~\ref{lm:f'} gives
$$g_\alpha'(z) Av = A\,g_\alpha'(z)\,v = \lambda\,Av.$$
Hence, $Av$ is also a $\lambda$-eigenvector. Clearly, $\{Av\,|\, A
\in S_{\LL{1}{}}\}$ spans \LL{1}{\perp}.

\end{proof}

The question now is whether there is always some parameter-choice
for which the eigenvalue of $\LL{1}{\perp}$ vanishes for all $z$.

\begin{prop} \label{prop:g4Crit}

For $n \geq 5$, there is a $4$-map $g$ whose critical set includes
the $\binom{n}{2}$-lines. Moreover, at each point on an
$\binom{n}{2}$-line, $g$ is critical in every direction away from
the line.

\end{prop}

\begin{proof}

To facilitate exposition, we work in the linear space $\HH^{n-1}$. For an
arbitrary member of the family $g_\alpha$, select a lift
$\widehat{g_\alpha}$ to $\HH^{n-1}$.  The line \LL{1}{} lifts to the plane
$\widehat{\LL{1}{}}$ parametrized by
\begin{equation} \label{eq:L1}
(x+y,\ \dots,\ ,x+y,(1-n)\,x+y,x+(1-n)\,y).
\end{equation}
Furthermore, the orthogonal complement $\widehat{\LL{1}{}}_\perp$
in $\HH^{n-1}$ is
$$\Biggl\{ \sum_{k=1}^{n-2} x_k = 0, x_{n-1}=x_n=0 \Biggr\}.$$

By symmetry, we can consider a single line. Using the parametrization
above, the pair of $n$-points on \LL{1}{} correspond to the lines $x=0$
and $y=0$. Meanwhile, the line specified by
$$x+y=0$$
determines an element in one of the special orbits of
$\binom{n}{2}$-points.

Proposition~\ref{prop:psi'} implies that the characteristic polynomial for
the jacobian of $\widehat{g_\alpha}$, when restricted to
$\widehat{\LL{1}{}}$, has the form
\begin{equation*}
\chi_{\widehat{g_\alpha}'}|_{\widehat{\LL{1}{}}} =
  \det{(t\,I_n - \widehat{g_\alpha}')}|_{\widehat{\LL{1}{}}} =\
    \bigl(t - A(x,y)\bigr)^{n-3} \bigl(t^2 - B(x,y)\,t + C(x,y)\bigr)
\end{equation*}
Note that the factor
$$t^2 + B\,t + C$$
is the characteristic polynomial of the jacobian of the map
$\widehat{g_\alpha}|_{\widehat{\LL{1}{}}}$.  Hence,
$$B(x,y) = B_1\,x^3 + B_2\,x^2 y + B_3\,x y^2 + B_4\,y^3$$
where the $B_i$ are linear in the parameters $\alpha_j$.  It
follows that
$$A(x,y) = A_1\,x^3 + A_2\,x^2 y + A_3\,x y^2 + A_4\,y^3$$
where the $A_i$ are linear in the $\alpha_j$.

The polynomial $A$ gives the eigenvalue in the off-line directions
in $\PP{}\HH^{n-1}$. The remaining factor corresponds to behavior
along \LL{1}{}. Our interest here is $A$. In particular, we want
to force it to vanish identically in $x$ and $y$.

Since $\chi_{\widehat{g_\alpha}'}$ is invariant under permutation
of coordinates, its restriction to $\widehat{\LL{1}{}}$ is
invariant under the interchange $x\leftrightarrow y$ (which
corresponds to the transposition $((n-1)\,n)$ in \eqref{eq:L1}).
Accordingly, $A_1=A_4$ and $A_2=A_3$ so that
\begin{align}
A&=\ A_1\,(x^3+y^3) + A_2\,xy(x+y) \nonumber \\
 &=\ (x+y)(A_1\,(x^2-xy+y^2) + A_2\,xy).
\end{align}
Thus, the $\binom{n}{2}$-points are automatically critical away from
\LL{1}{}.

We want to solve the linear equations $A_1=0$ and $A_2=0$ in the
three parameters $\alpha_1,\alpha_2,\alpha_3$.  Such a system has
non-trivial solutions that give $A=0$ for all $x,y$.

\end{proof}

\noindent \textbf{Remarks}

\begin{enumerate}

\item The price of $A=0$ is at most two parameters.  With the third
parameter, we can only normalize the map.  In the discussion that
follows, we will discover that the cost is two parameters so that
the resulting map is unique.

\item Forcing $A_1=0$ is tantamount to making the pair of $n$-points on
\LL{1}{}  critical.  This ``two-birds-with-one-stone" effect is
what makes the procedure successful.

\item Ostensibly, this argument is consistent with our obtaining a map that
blows up at the $n$-points.  At an $n$-point such a map would be
critical in the ``radial" direction in which $\HH^{n-1}$ projects
to $\PP{}\HH^{n-1}$.  Such a circumstance would force $C=0$ when
$xy=0$. Can the maps be critical in other directions as well?  The
preceding results and proofs provide for explicit calculation of
the special $4$-map.  As a consequence, we see that the map does
not blow up at the $n$-points.  Furthermore, we derive the form of
the map on an $\binom{n}{2}$-line.

\end{enumerate}

At first, we use the parametrization
$$[x,\dots,x,y,(2-n)\,x -y]$$
for \LL{1}{}.  Restricting the basic invariants to \LL{1}{} gives
\begin{equation}
\tld{F_k} =\ F_k|_{\LL{1}{}} =\
            (n-2)\,x^k + y^k + ((2-n)\,x-y)^k.
\end{equation}

As for basic maps, note that
\begin{equation} \label{eq:fK}
f_k=[(f_k)_1,\dots,(f_k)_n]\ \text{with}\
(f_k)_\ell=F_k-n\,x_\ell^k.
\end{equation}
Thus,
$$
\tld{f_k} =\ f_k|_{\LL{1}{}} =\ \Bigl[\tld{F_k}-n\,x^k,\dots,
  \tld{F_k}-n\,x^k,\tld{F_k}-n\,y^k,\tld{F_k}-n\,((2-n)\,x-y)^k \Bigr]
$$
and we can express the homogeneous map \emph{on} \LL{1}{} as
$$
\phi_k :
 [x,y] \longrightarrow \Bigl[\tld{F_k}-n\,x^k,\tld{F_k}-n\,y^k \Bigr].
$$

Restricting the family of $4$-maps to \LL{1}{} gives
$$
\tld{g_\alpha} = \alpha_1\,\tld{f_4}
+ \alpha_2\,\tld{F_2}\,\tld{f_2} + \alpha_3\,\tld{F_3}\,\tld{f_1}
$$
which, as a map on the line, is
\begin{align*}
\gamma_\alpha[x,y] =&\ \Bigl[
  \alpha_1\,(\tld{F_4}-n\,x^4) + \alpha_2\,\tld{F_2}\,(\tld{F_2}-n\,x^2)
  - n\,\alpha_3\,x\,\tld{F_3},\\
&\ \alpha_1\,(\tld{F_4}-n\,y^4) + \alpha_2\,\tld{F_2}\,(\tld{F_2}-n\,y^2)
  - n\,\alpha_3\,y\,\tld{F_3}
\Bigr].
\end{align*}

We can now determine three linear conditions on the $\alpha_i$ that
correspond to

\begin{enumerate}

\item normalizing the map so that the $n$-points are fixed in the affine
sense---hence, they are not blown up

\item making the map critical in every direction at the $n$-points

\item making the map critical in every off-line direction along the
$\binom{n}{2}$-lines.

\end{enumerate}

Consider the cases in turn.

\begin{enumerate}

\item Since $[x,y]=[1,1]$ corresponds to an $n$-point, specify that
\begin{equation} \label{eq:PnNorm}
1=(\gamma_\alpha[1,1])_1 = \bigl( n^2-2\,n+2)\,\alpha_1 +
(n-1)\,n\,\alpha_2 + (n-2)\,n\,\bigl((n-1)\,n\,\alpha_3
\bigr).
\end{equation}

\item For the dehomogenized map
$$f(z) = \frac{(\gamma_\alpha[z,1])_1}{(\gamma_\alpha[z,1])_2},$$
set
$$f'(1)=0.$$
This requires
\begin{equation} \label{eq:PnCrit}
- 4\,\alpha_1 - 2\,(n-1)\,n\,\alpha_2 + (n-2)\,(n-1)\,n\,\alpha_3 = 0
\end{equation}
which amounts to $A_1=0$ in the proof of Proposition~\ref{prop:g4Crit}.

\item  To arrive at $A_2=0$, evaluate the jacobian
$\widehat{g_\alpha}'(x)$ at a point $P$ on \LL{1}{} other than an
$n$-point or a $\binom{n}{2}$-point---say, the point corresponding to
$[x,y]=[1,0]$. Now, apply $g_\alpha'(P)$ to the \emph{eigenvector}
$$v=(1,-1,0,\dots,0)$$
in $\widehat{\LL{1}{}}_\perp$.  Since $v$ is an eigenvector, we can
consider just the first component of $g_\alpha'(P) v$.  This component is
$$g_\alpha'(P)_{1,1}-g_\alpha'(P)_{1,2}$$
where
\begin{align*}
g_\alpha'(P)_{1,1} &=\ -4\,(n-1)\,x^3\,\alpha_1 -
  2\,x\,(n\,x^2 +(n-2)\,\tld{F_2}(1,0)) \,\alpha_2 -
  n\,(3\,x^3 + \tld{F_3}(1,0))\,\alpha_3\\
g_\alpha'(P)_{1,2} &=\ 4\,x^3\,\alpha_1 -
  2\,x\,(n\,x^2 - 2\,\tld{F_2}(1,0))\,\alpha_2 - 3\,n\,x^3\,\alpha_3.
\end{align*}
Thus, the condition for $A_2=0$ is
\begin{equation} \label{eq:L1Crit}
n\,\bigl(-4\,\alpha_1 -
  2\,(n-2)\,(n-1)\,\alpha_2 + (n-3)\,(n-2)\,(n-1)\,\alpha_3
\bigr) = (g_\alpha'(P) v)_1 = 0.
\end{equation}

\end{enumerate}

Solving \eqref{eq:PnNorm}, \eqref{eq:PnCrit}, and
\eqref{eq:L1Crit} yields a unique $4$-map $g=g_\alpha$ with
\begin{align} \label{eq:alpha}
\alpha_1&=\ \frac{1}{(n-4)\,n^3} \\
\alpha_2&=\ \frac{-6}{(n-4)(n-1)n^4} \nonumber \\
\alpha_3&=\ \frac{-8}{(n-4)(n-2)(n-1)n^4} \nonumber .
\end{align}

Natural coordinates in which to express $g$ as a map on \LL{1}{}are those
of \eqref{eq:L1} where $x$ and $y$ are symmetrical.  The result is
\begin{align*}
[x,y] \longrightarrow &\
[(\gamma_\alpha[x+y,(1-n)\,x+y])_1 -(\gamma_\alpha[x+y,(1-n)\,x+y])_2,\\
&\ (n-1)\,(\gamma_\alpha[x+y,(1-n)\,x+y])_1
-(\gamma_\alpha[x+y,(1-n)\,x+y])_2].
\end{align*}
Using the inhomogeneous coordinate $z=\tfrac{x}{y}$ the map
becomes
$$z \longrightarrow -\frac{(n-1)\,z-4}{4\,z-(n-1)}$$
in agreement with Section~\ref{sec:4map}.  As for dynamics the
respective pair of $n$-points are $0$ and $\infty$.  Since
$n-1\geq~4$, the restricted map has $0$ and $\infty$ as its only
attractor.  The respective basins are $\{|z|<1\}$ and $\{|z|>1\}$.

\subsection{Another description}

Using $n-1$ homogeneous coordinates, generalize to \CP{n-1} the $v$
coordinates that describe the \G{8!}-equivariant $g_4$. These place the
$n$-points at
$$p_1=[1,0,\dots,0],\ \dots\ ,\ p_{n-1}=[0,\dots,0,1],\ p_n=[1,\dots,1].$$
The coordinate change is given by
$$x=Pv \qquad v=Qx$$
where
$$
P=(p_{ij})=
\begin{cases} 1-n&i=j\\ 1&i \neq j \end{cases} \qquad \text{for}\
1 \leq i \leq n,\ 1 \leq j \leq n-1
$$
and $Q$ is the ``inverse" of $P$:
$$
Q=(q_{ij})=
\begin{cases} -1&i=j\\ 0&i \neq j,\ j<n\\1&j=n \end{cases} \qquad
\text{for}\ 1 \leq i \leq n-1,\ 1 \leq j \leq n.
$$

Using $[v_1,\dots,v_{n-1}]$ for this system, we define $S_k$ to be
the $k$th elementary symmetric function in $n-2$ variables and the
coordinates
$$\widehat{v_k} = (\dots,v_{k-1},v_{k+1},\dots)$$
complementary to $v_k$.  The stabilizer \G{(n-1)!} of $p_n$ is the
\Sym{n-1} group of permutations of the $v_k$. The order-$2$
transformation $Z_n$ that exchanges $p_1$ and $p_n$ while fixing
the remaining $p_k$ generates \G{n!} over \G{(n-1)!}.  Note that
$\binom{n-1}{2}$ of the $\binom{n}{2}$-lines consist of points for
which all but two coordinates vanish, while the remaining $n$
lines have points with all but one coordinate equal.

Let
$$g(x)=[g_1(x), \dots,g_n(x)].$$
To compute the special $4$-map
$$\gamma(v)=Q(g(Pv))=[g_n(Pv)-g_1(Pv), \dots, g_n(Pv)-g_{n-1}(Pv)]$$
in $v$ coordinates, we need to find only the first component of
$\gamma$. Permutation symmetry in $v$ tends to the remaining
components.  Note that
$$Pv=\begin{pmatrix}
(1-n)\,v_1 + S_1(\widehat{v_1})\\
v_1 + S_1(\widehat{v_1}) - n\,v_2\\
\vdots\\
v_1 + S_1(\widehat{v_1}) - n\,v_{n-1}\\
v_1 + S_1(\widehat{v_1})
\end{pmatrix}
=\begin{pmatrix}
u_1-n\,v_1\\
u_1-n\,v_2\\
\vdots\\
u_1-n\,v_{n-1}\\
u_1
\end{pmatrix}
$$
where $u_1=v_1+S_1(\widehat{v_1})$.  Application of \eqref{eq:fK} gives
\begin{align*}
\gamma_1(v)=&\ g_n(Pv)-g_1(Pv)\\
=&\ \Bigl(
  \alpha_1\,\bigl(F_4(Pv) - n\,u_1^4 \bigr)
  + \alpha_2\,\bigl(F_2(Pv)^2 - n\,F_2(Pv)u_1^2 \bigr)
  + n\,\alpha_3\,F_3(Pv)u_1
\Bigr)\\
&- \Bigl(
  \alpha_1\,\bigl(F_4(Pv) - n\,(u_1-n\,v_1)^4 \bigr)
  + \alpha_2\,\bigl(F_2(Pv)^2 - n\,F_2(Pv)(u_1-n\,v_1)^2
  \bigr)\\
  &\ + n\,\alpha_3\,F_3(Pv)(u_1-n\,v_1)
\Bigr)\\
=&\ n\,\Bigl(
  \alpha_1\,\bigl((u_1-n\,v_1)^4-u_1^4 \bigr)
  +\alpha_2\,F_2(Pv)\bigl((u_1-n\,v_1)^2-u_1^2 \bigr)\\
  &\ +\alpha_3\,F_3(Pv)\bigl((u_1-n\,v_1)-u_1 \bigr)
\Bigr)\\
=&\ -n^2\,v_1\,\Bigl(
  \alpha_1\,\bigl(
  (u_1-n\,v_1)^3 + (u_1-n\,v_1)^2 u_1 + (u_1-n\,v_1) u_1^2 + u_1^3
  \bigr)\\
  &\ + \alpha_2\,F_2(Pv) (2\,u_1-n\,v_1)
  + \alpha_3\,F_3(Pv)
  \Bigr).
\end{align*}
Straightforward calculation yields
\begin{align*}
F_2(Pv)=&\
u_1^2 - n\,v_1^2 - n\,S_1(\widehat{v_1})^2 + 2\,n\,S_2(\widehat{v_1})\\
F_3(Pv)=&\ 2\,u_1^3 - 3\,n\,u_1\,v_1^2 + n^2\,v_1^3 -
3\,n\,u_1\,S_1(\widehat{v_1})^2 + n^2\,S_1(\widehat{v_1})^3 \\
&\ + 6\,n\,u_1\,S_2(\widehat{v_1}) -
3\,n^2\,S_1(\widehat{v_1})\,S_2(\widehat{v_1}) +
3\,n^2\,S_3(\widehat{v_1}).
\end{align*}
Using \eqref{eq:alpha}, the substitution for $u_1$, and
permutation in the $v_i$, the map takes the form
\begin{equation} \label{eq:g4maps}
\gamma(v) = [v_1\,T_1(v),\dots,v_{n-1}\,T_{n-1}(v)]
\end{equation}
where
$$
T_k(v) = v_k^3 - a_2\,v_k^2\,S_1(\widehat{v_k}) +
a_3\,v_k\,S_2(\widehat{v_k}) - a_4\,S_3(\widehat{v_k})
$$
and
$$
a_2=\frac{4}{n-1} \qquad a_3=\frac{12}{(n-1)(n-2)}\qquad
a_4=\frac{24}{(n-1)(n-2)(n-4)}.
$$
Evidently, the maps in \eqref{eq:g4maps} are
\G{(n-1)!}-equivariant. They also satisfy
$$Z_n \circ g = g \circ Z_n \qquad \text{for all $n \geq 5$}.$$

\subsection{Revisiting the quintic}

In solving the quintic, \cite{quintic} harnesses the dynamics of a
degree-$6$ map whose behavior is similar to the maps treated in
the present paper. The \G{5!} equivariant $4$-map is also a good
candidate for inclusion in a quintic-solving algorithm.  Having
lower degree, its global dynamics might be more tractable.  For
instance, this map has a kind of \emph{critical finiteness} that
other \G{n!} $4$-maps do not share.  I plan to examine this map in
more dynamical detail in an upcoming paper.

\section{Basin portraits} \label{app:basins}

The plots that follow are productions of the program
\emph{Dynamics 2} running on a Dell Dimension XPS with a Pentium
II processor. Its BA and BAS routines produced the pictures. (See
the manual \cite{NY}.) Briefly, each procedure divides the screen
into a grid of cells and then colors each cell according to which
attracting point its trajectory approaches. If it finds no such
attractor after a specified number of iterations---usually $60$,
the cell is black. The BA algorithm looks for an attractor whereas
BAS requires the user to specify a candidate attracting set of
points. Each portrait exhibits the highest resolution
available---a $720 \times 720$ grid.

All of the images show $g_4$ restricted to either a \CP{1} or
\RP{2} that it preserves.  Some restricted maps have attracting
sites that are not the $8$-points. However, none of the detected
``restricted attractors" other than the $8$-points themselves are
\emph{overall} attractors with $6$-dimensional basins.

\subsection*{\fig{fig:g4L1_168}} When restricted to the line \LL{1}{168},
$g_4$ has only trivial symmetry.  Consequently, there is no
natural choice of coordinates in which to express the restriction
to the line. With an $8$-point at $0$ and a $28$-point \q{28}{ij}
at $\infty$, the map takes the form
$$z \longrightarrow -\frac{z^3\,(3\,z+4)}{4\,(z^2+8\,z+14)}$$
The two basins that appear are associated with the $8$-point and
$28$-point. The latter is an ``equatorial" saddle point on the $28$-line
\LL{1}{28_{ij}}. The basin in \LL{1}{168} is a slice of its
$5$-dimensional stable manifold.  The reflective symmetry that appears is
due to the map's anti-holomorphic equivariance.  (See
Section~\ref{sec:4map}.)

Not pictured is the portrait for \LL{1}{210} on which the map
takes the form
$$z \longrightarrow -z^2\,\frac{3\,z+1}{z+3}.$$
As in the case of a $28$-line, this map has two basins: $\{|z|<1\}$ and
$\{|z|>1\}$.  The attracting points are of the types
$$[1,1,-1,-1,0,0,0,0] \qquad \text{and} \qquad [1,1,1,1,-1,-1,-1,-1].$$
Overall, these behave as saddles.  In fact, at the latter point, $g_4$
blows up onto the associated hyperplane
$$\{x_1+x_2+x_3+x_4=0\} \cap \{x_5+x_6+x_7+x_8=0\}.$$

\subsection*{\fig{fig:g4L1_332}} This shows $g_4$ on \LL{1}{280} in the
form
$$z \longrightarrow \frac{8\,z^3}{5\,z^4+6\,z^2-3}.$$
The basin at $0$ is due to a $28$-point \p{28}{ij}.  This point
belongs to \LL{1}{28_{ij}} and is repelling on that line.  Thus,
the basin pictured here is a $1$-dimensional slice of the
$5$-dimensional stable manifold attached to the point. At $\pm 1$
we see petals due to rationally indifferent points of type
$$[3,3,3,3,3,-5,-5,-5].$$
The indifferent local behavior at these points is evident in this
and several subsequent images.

\subsection*{\fig{fig:g4L1_431}} On \MM{1}{280}, $g_4$ is a polynomial map
expressible by
$$z \longrightarrow -z\,(7\,z^2-5\,z+1).$$
Here, an $8$-point is at $\infty$ (the dark blue basin) and, once
again, a rationally indifferent point of type
$$[3,3,3,3,3,-5,-5,-5]$$
appears at $0$.

\subsection*{\fig{fig:g4R2_56}} We see the \Sym{3}-symmetric restriction to
the \RP{2} intersection of \LL{2}{56} and \RR.  The three basins
result from a triple of $8$-points arranged symmetrically on the
unit circle.  A line of reflective symmetry passes through each of
the $8$-points and $(0,0)$---which corresponds to a point of type
$$[3,3,3,3,3,-5,-5,-5].$$  These lines are \RP{1} intersections of an
\LL{1}{168} with the \RP{2}.  Each \RP{1} contains points in the
basin of an $8$-point and as well as points in the boundary
between the basins of the other two $8$-points.  This interval
lies in the basin of a $28$-point \q{28}{ij} as seen in
\fig{fig:g4L1_168}.

\subsection*{\fig{fig:g4R2_105}} The image displays the basins of
the \Sym{4}-symmetric map on the intersection of \LL{2}{105} and
\RR.  Each \LL{2}{105} is canonically associated with a \CP{3},
\LL{3}{105}, whose points have coordinates that come in four
mutually negative pairs:
$$\begin{array}{lcl}
\{x_i=x_j\}\cap \{x_k=x_\ell\} & \longleftrightarrow &
\{x_i=-x_j\}\cap \{x_k=-x_\ell\}\\
\cap \{x_m=x_n\} \cap \{x_p=x_q\}&
& \cap \{x_m=-x_n\} \cap \{x_p=-x_q\}.
\end{array}$$
Under $g_4$, \LL{3}{105} blows down to \LL{2}{105}. Furthermore,
each point in the orbit of
$$[1,-1,1,-1,1,-1,1,-1]$$
belongs to $24$ of the \LL{3}{105} so that the blowing-down of each \CP{3}
forces all coordinates of the image to be equal. Hence, they must all
vanish.

To describe things explicitly, take the plane
$$
\LL{2}{105} = \{x_1=x_2\}\cap \{x_3=x_4\} \cap \{x_5=x_6\} \cap
\{x_7=x_8\}
$$
parametrized by
$$
[x,y,z] \longrightarrow
[x+y+z,x+y+z,x-y-z,x-y-z,-x-y+z,-x-y+z,-x+y-z,-x+y-z].
$$
There is a $3$-point orbit at
$$[1,0,0],\ [0,1,0],\ [0,0,1]$$
corresponding to the points
$$[1,1,1,1,-1,-1,-1,-1],\ [1,1,-1,-1,-1,-1,1,1],\ [1,1,-1,-1,1,1,-1,-1]$$
and a $4$-point orbit at $[\pm 1,\pm 1,1]$ corresponding to the
$28$-points
\begin{align*}
\q{28}{12}&=\ [-3,-3,1,1,1,1,1,1],\ \q{28}{34}=[1,1,-3,-3,1,1,1,1]\\
\q{28}{56}&=\ [1,1,1,1,-3,-3,1,1],\
\q{28}{78}=[1,1,1,1,1,1,-3,-3].
\end{align*}
Associated with the $3$-points are the lines
$$\{x=0\},\ \{y=0\},\ \{z=0\}$$
corresponding to intersections with three of the \LL{3}{105}:
$$[u,u,-u,-u,v,v,-v,-v],\ [u,u,v,v,-v,-v,-u,-u],\ [u,u,v,v,-u,-u,-v,-v].$$

In these coordinates, the map has the expression
$$
[x,y,z] \longrightarrow
 [yz(4\,x^2+y^2+z^2),xz(x^2+4\,y^2+z^2),xy(x^2+y^2+4\,z^2)].
$$
The attractor here seems to consist of four $28$-points \q{28}{ij} at
$(\pm 1,\pm 1)$.  The points \q{28}{ij} are on \LL{1}{28_{ij}} and, thus,
repel in directions away from the picture-plane.  Here we see off-line
directions relative to \LL{1}{28_{ij}}.  Each
\q{28}{ij} lies on $15$ of the \LL{2}{105} and their ``restricted"
basins in the \CP{2} belong to their stable manifolds as well as
to the \emph{overall} Julia set of $g_4$. Hence, these pieces of
the Julia set have zero measure in \CP{6}. As for real dynamics,
the basins appear to be the four quadrants
$$\{x>0,y>0\},\ \{x<0,y>0\},\ \{x<0,y<0\},\ \{x>0,y<0\}$$
which are forward and, thus, totally invariant.  Accordingly, the
coordinate axes form the basin boundaries.

Finally, each coordinate axis blows down to its companion point while a
$3$-point blows up onto its associated axis.

\subsection*{\fig{fig:g4R2_280}} For the $\Z{2}^2$ map on the
intersection of \LL{2}{280} and \RR{} the two prominent basins
belong to a pair of $8$-points at $(\pm 1,0)$. The other two
attracted regions are associated with a pair of indifferent points
(both eigenvalues are $-1$) of type
$$[3,3,3,3,3,-5,-5,-5].$$

\subsection*{\fig{fig:g4R2_420}} On the \Z{2}-symmetric restriction to the
intersection of \LL{2}{420} and \RR, two $8$-points at $(\pm 1,0)$
and a \q{28}{ij} at $(0,-1)$ account for the basins.

\subsection*{\fig{fig:g4R2_840}} As for the behavior of the \Z{2}-symmetric map on
the intersection of \LL{2}{840} and \RR, there is an $8$-point at
$(0,0)$ and a pair of $28$-points \q{28}{ij}, \q{28}{k\ell} at
$(\pm 1,1)$. Trajectories starting in the thin vertical black
strip near the bottom middle require too many iterations in order
to converge to one of the three attractors for the BA routine to
detect attraction. Experimental results suggest that this region
is filled with rather small components of the three basins.

\begin{figure}[ht]

\resizebox{\basinWidth}{!}{\includegraphics{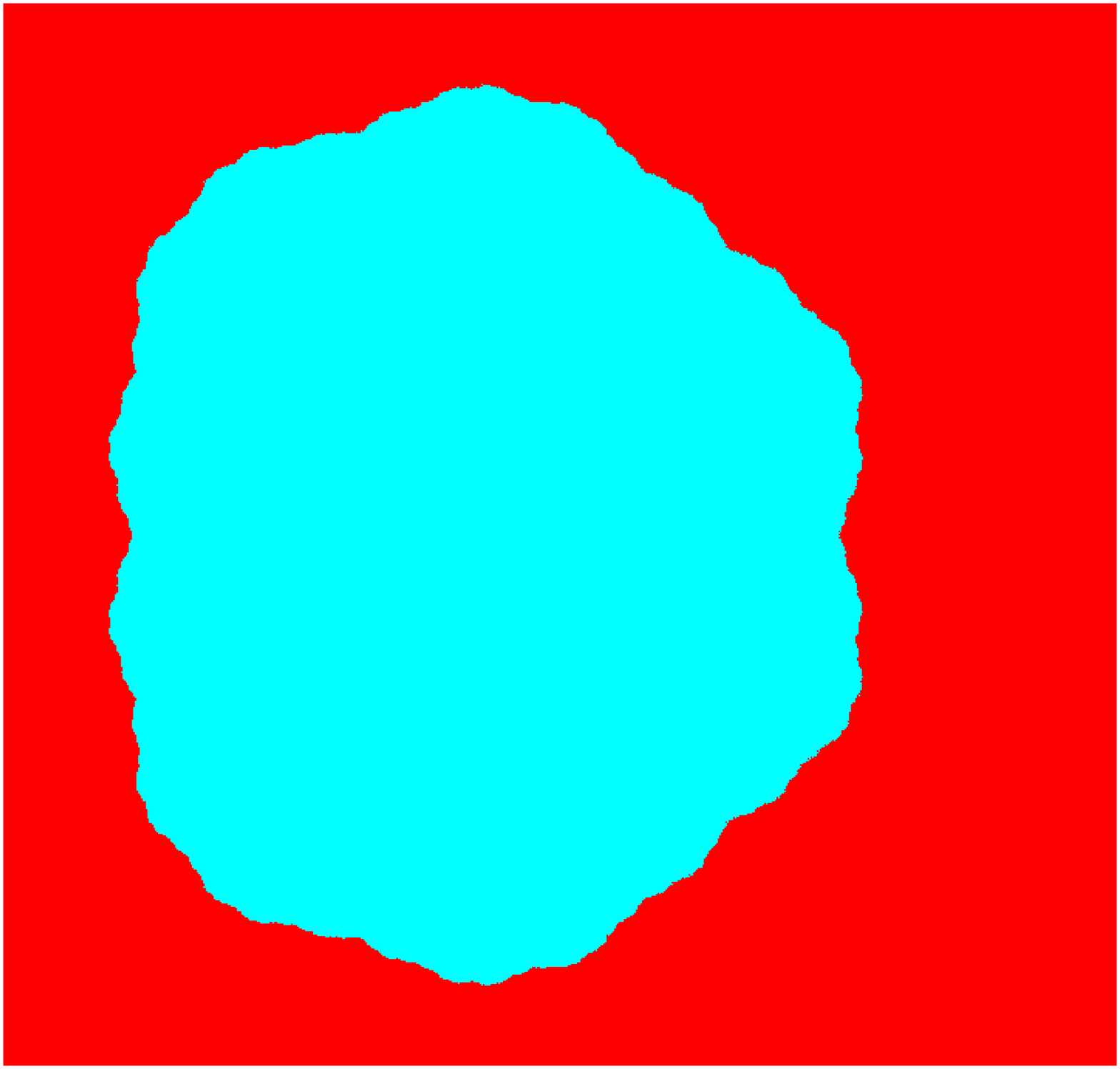}}

\caption{Dynamics of $g_4$ on \LL{1}{168}}

\label{fig:g4L1_168}

\end{figure}
\begin{figure}

\resizebox{\basinWidth}{!}{\includegraphics{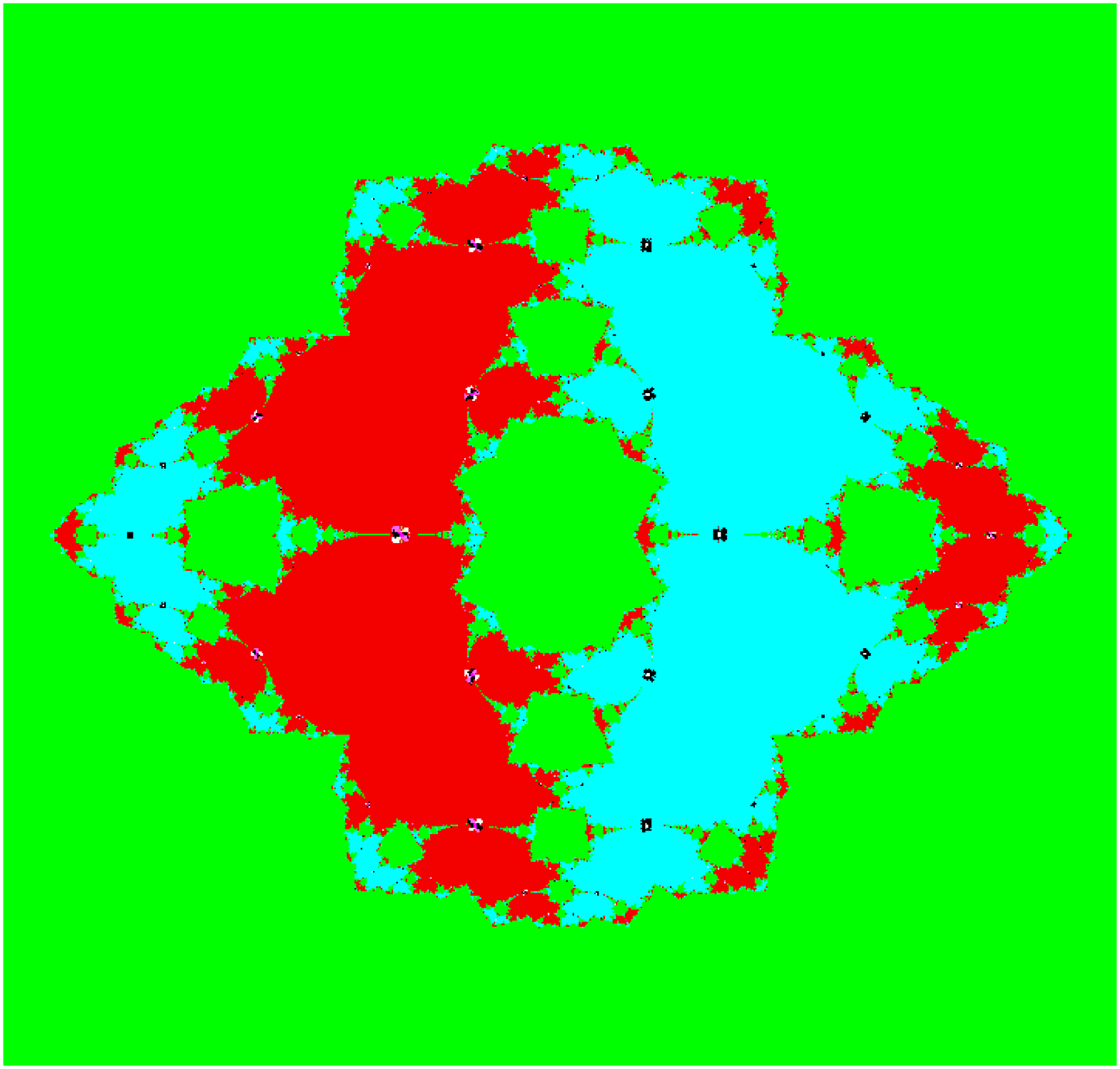}}

\caption{Dynamics of $g_4$ on \LL{1}{280}}

\label{fig:g4L1_332}

\end{figure}
\begin{figure}[ht]

\resizebox{\basinWidth}{!}{\includegraphics{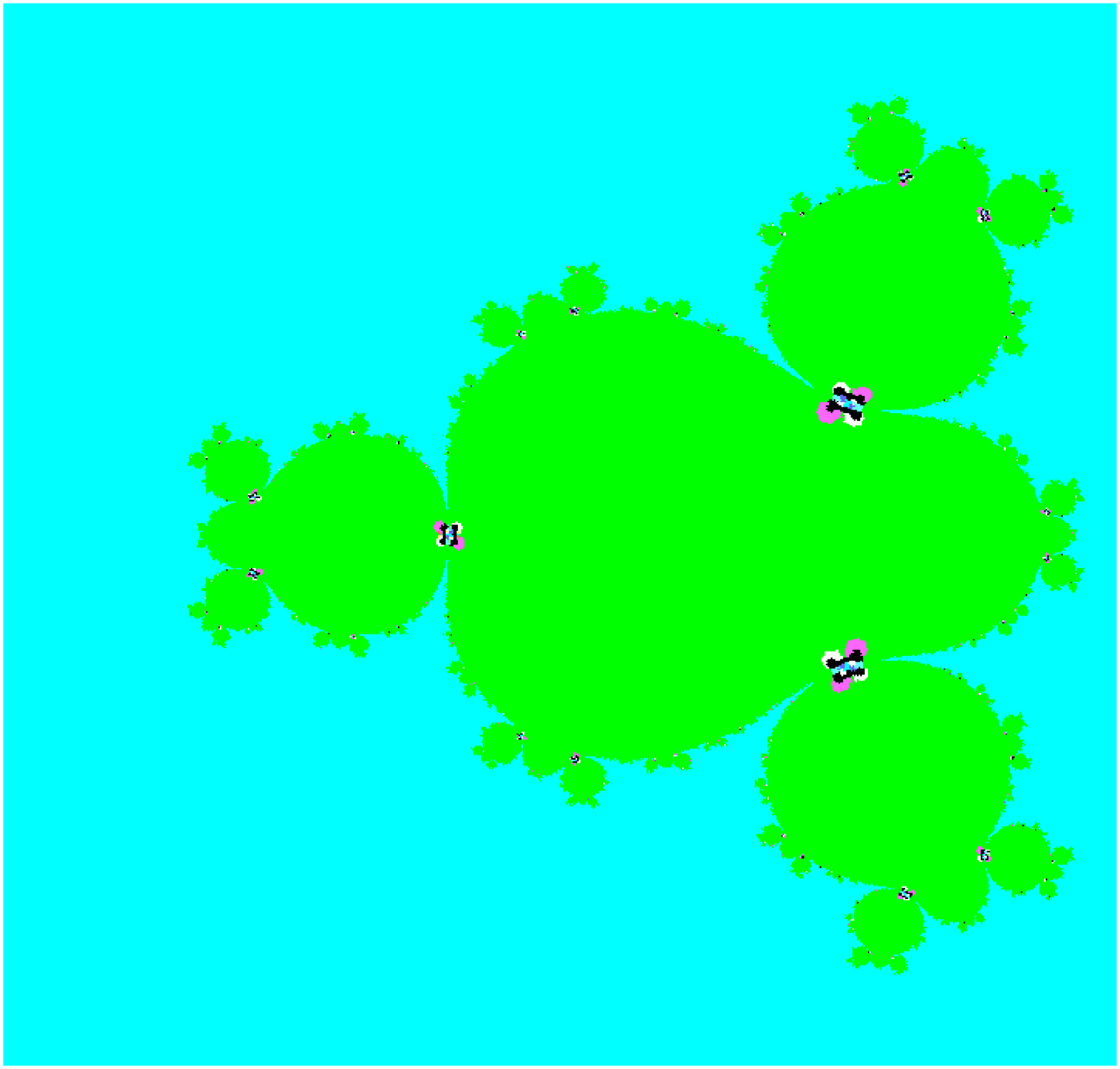}}

\caption{Dynamics of $g_4$ on \MM{1}{280}}

\label{fig:g4L1_431} 

\end{figure}
\begin{figure}[ht]

\resizebox{\basinWidth}{!}{\includegraphics{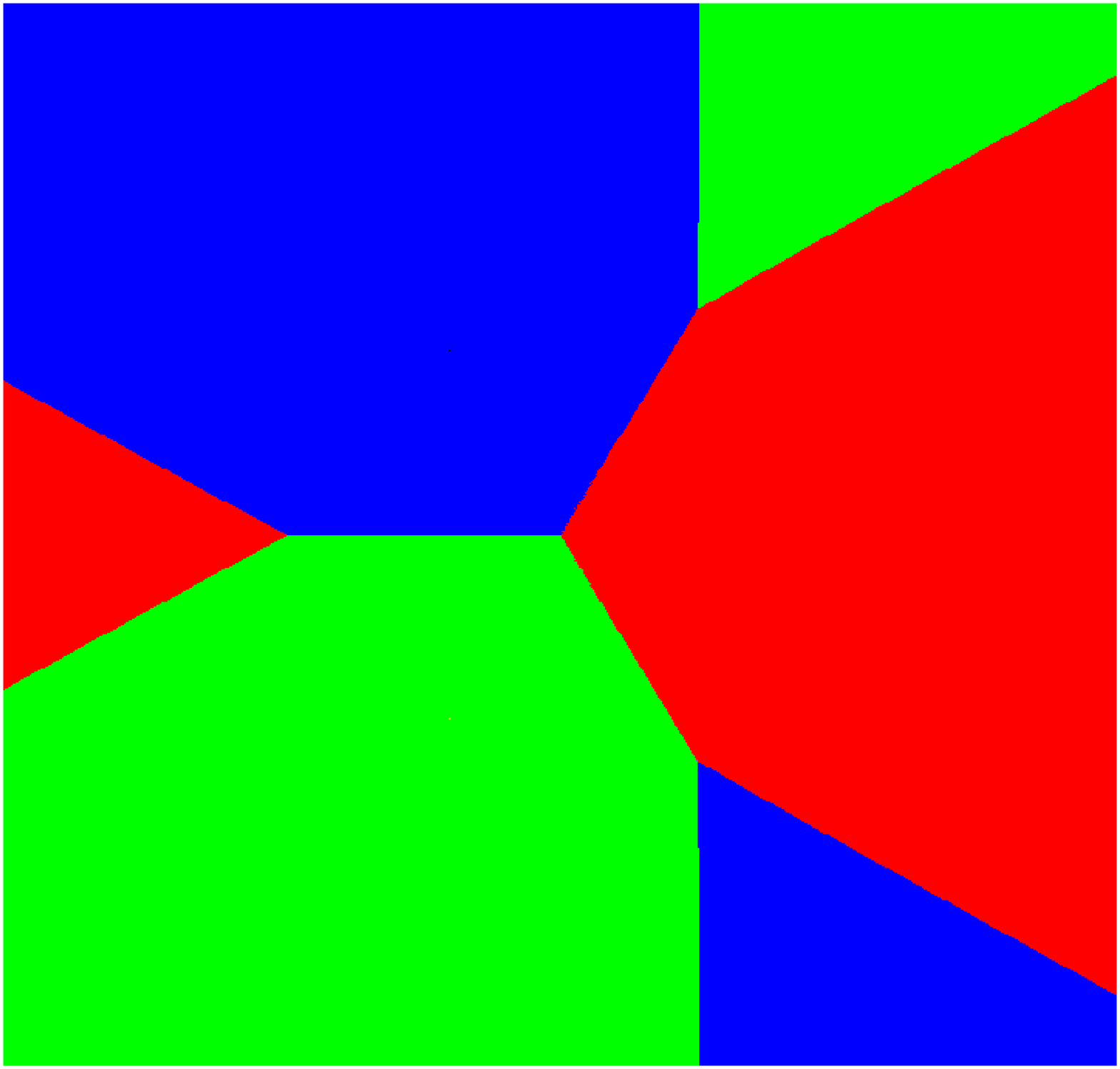}}

\caption{Dynamics of $g_4$ on $\LL{2}{56} \cap \RR$}

\label{fig:g4R2_56}

\end{figure}
\begin{figure}[ht]

\resizebox{\basinWidth}{!}{\includegraphics{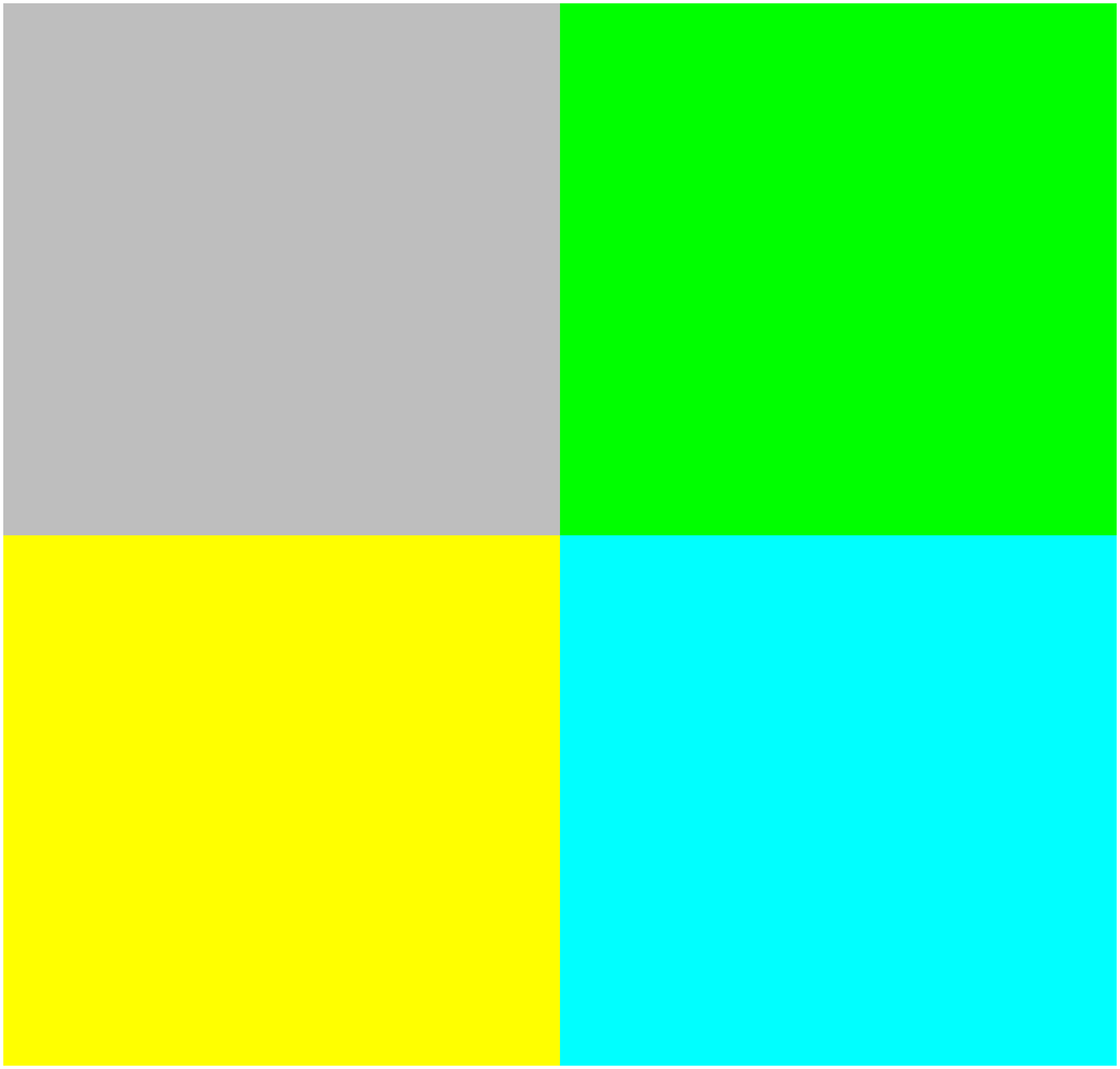}}

\caption{Dynamics of $g_4$ on $\LL{2}{105} \cap \RR{}$}

\label{fig:g4R2_105}

\end{figure}
\begin{figure}[ht]

\resizebox{\basinWidth}{!}{\includegraphics{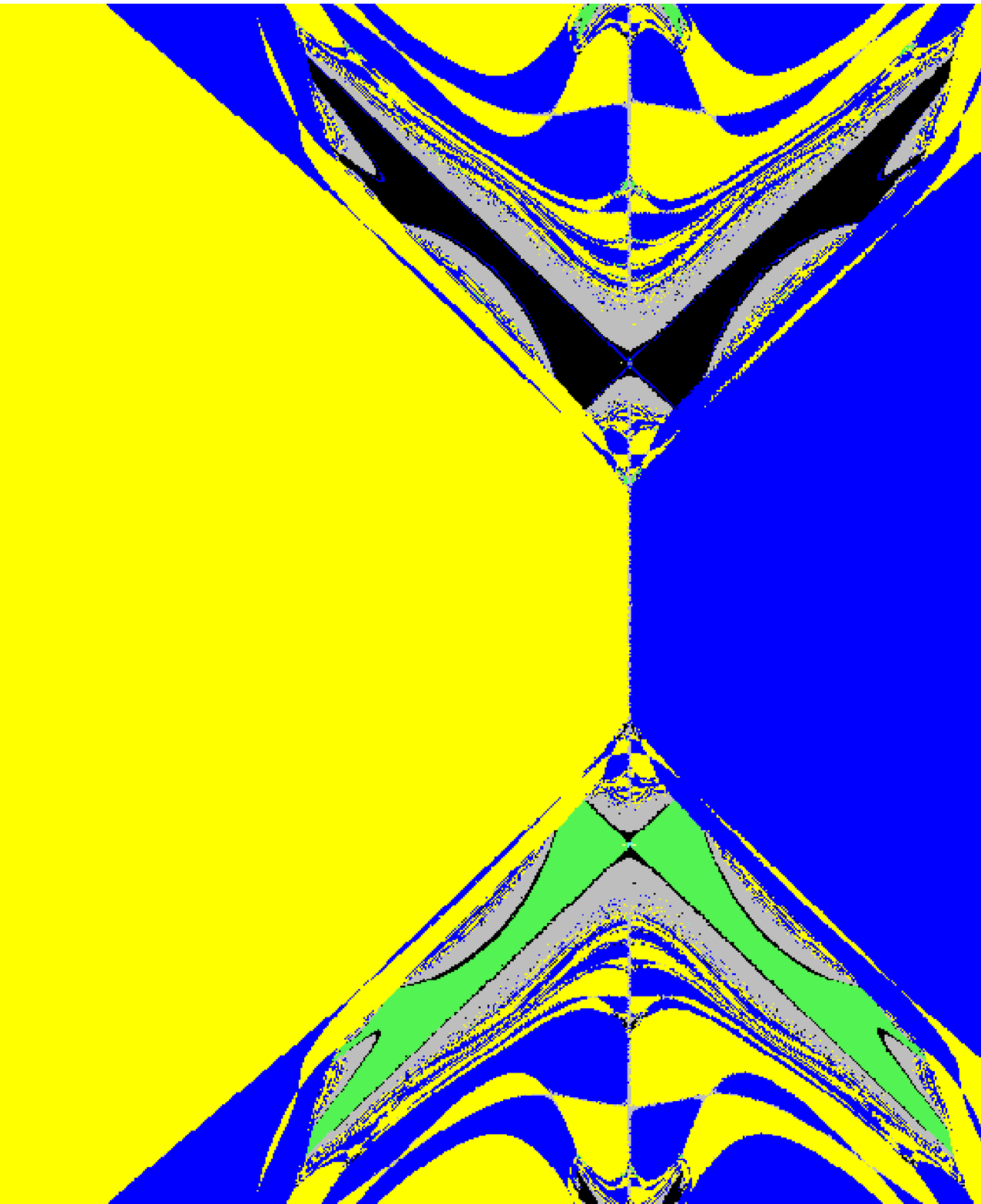}}

\caption{Dynamics of $g_4$ on $\LL{2}{280} \cap \RR$}

\label{fig:g4R2_280}

\end{figure}
\begin{figure}[ht]

\resizebox{\basinWidth}{!}{\includegraphics{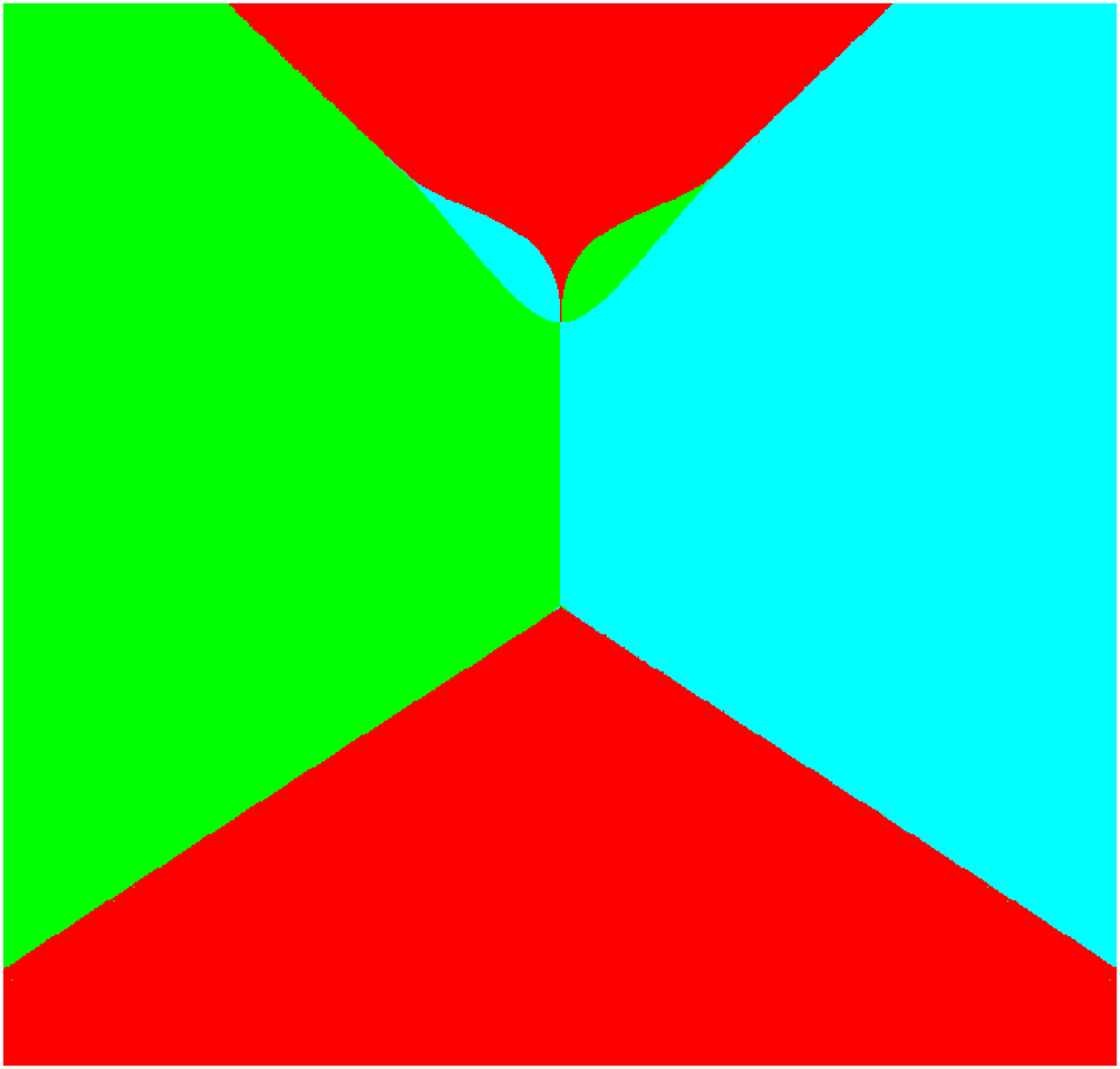}}

\caption{Dynamics of $g_4$ on $\LL{2}{420} \cap \RR$}

\label{fig:g4R2_420}

\end{figure}
\begin{figure}[ht]

\resizebox{\basinWidth}{!}{\includegraphics{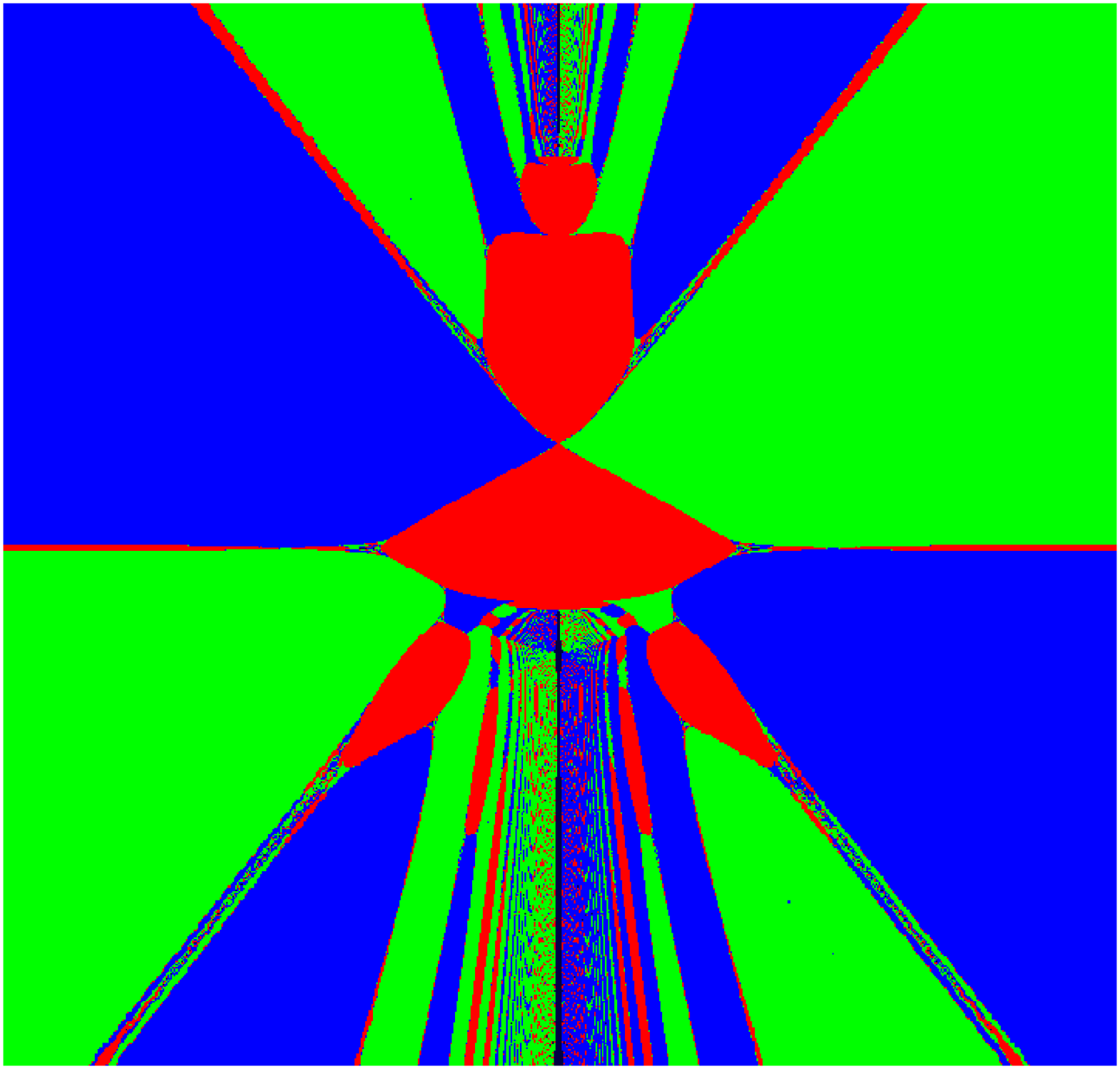}}

\caption{Dynamics of $g_4$ on $\LL{2}{840} \cap \RR$}

\label{fig:g4R2_840}

\end{figure}

\end{document}